\documentclass[12pt,a4paper]{article}
\usepackage{amsmath,amssymb,amsthm,graphicx,color}
\usepackage[left=1in,right=1in,top=1in,bottom=1in]{geometry}
\usepackage{cite}
\usepackage[british]{babel}
\usepackage{enumerate}
\usepackage{hyperref} 
\usepackage{bm}
\usepackage{comment}

\hypersetup{
    colorlinks=false,
    pdfborder={0 0 0},
}

\newtheorem{theorem}{Theorem}[section]
\newtheorem{corollary}[theorem]{Corollary}
\newtheorem{proposition}[theorem]{Proposition}
\newtheorem{lemma}[theorem]{Lemma}

\numberwithin{equation}{section}

\theoremstyle{definition}

\theoremstyle{remark}

\newtheorem*{remark*}{Remark}

 \allowdisplaybreaks

\newcommand{\1}[1]{{\mathbf 1}{\{#1\}}}
\newcommand{\2}[1]{{\mathbf 1}{(#1)}}

\newcommand{\R}{{\mathbb R}}
\newcommand{\Z}{{\mathbb Z}}
\newcommand{\N}{{\mathbb N}}
\newcommand{\ZP}{{\mathbb Z}_+}
\newcommand{\RP}{{\mathbb R}_+}

\DeclareMathOperator{\Exp}{\mathbb{E}}
\renewcommand{\Pr}{{\mathbb P}}

\newcommand{\eps}{\varepsilon}

\newcommand{\re}{{\mathrm{e}}}
\newcommand{\rc}{{\mathrm{c}}}
\newcommand{\ud}{{\mathrm d}}

\newcommand{\cF}{{\mathcal F}}
\newcommand{\cG}{{\mathcal G}}

\newcommand{\as}{\ \text{a.s.}}
\newcommand{\io}{\ \text{i.o.}}

\newcommand{\tod}{\overset{\text{d}}{\longrightarrow}}

\newcommand{\bigmid}{\; \bigl| \;}
\newcommand{\Bigmid}{\; \Bigl| \;}

\newcommand{\eqd}{\overset{d}{=}}

\makeatletter
\def\namedlabel#1#2{\begingroup  
    (#2)%
    \def\@currentlabel{#2}%
    \phantomsection\label{#1}\endgroup
}
\makeatother

\begin{document}

\title{The critical greedy server on the integers is recurrent}
\author{James R.\ Cruise\footnote{Department of Actuarial Mathematics \& Statistics, Heriot-Watt University, Riccarton, Edinburgh EH14 4AS.} \and Andrew R.\ Wade\footnote{Department of Mathematical Sciences, Durham University, South Road, Durham DH1 3LE.}}

\date{\today}
\maketitle

\begin{abstract}
Each site of $\Z$ hosts a queue with arrival rate $\lambda$. A single server, starting at the origin,
serves its current queue at rate $\mu$ until that queue is empty, and then moves to the longest neighbouring queue.
In the critical case $\lambda = \mu$,
we show that the server returns to every site infinitely often. We also give a sharp iterated logarithm
result for the server's position. Important ingredients in the proofs are that the
times between successive queues being emptied exhibit doubly exponential growth,
and that the  probability that the server changes its direction is asymptotically equal to $1/4$.
\end{abstract}

\medskip


\medskip


\section{Introduction and main results}
\label{sec:intro}

The following continuous-time stochastic model was introduced by Kurkova and Menshikov~\cite{km}.
Each site of the one-dimensional integer lattice $\Z$
is associated with a queue. Each queue has an independent Poisson arrival
stream of rate $\lambda \in (0,\infty)$.
The system has a single server, which starts at the origin at time $0$. The server serves
the queue at its current site exhaustively at rate $\mu \in (0,\infty)$.
If the queue at the current site is empty at time $t \geq 0$, the server moves to one of the two
neighbouring sites using a greedy policy: it chooses to move to the site
with the longest queue (measured at time $t$), randomly breaking any tie.
The server moves (deterministically) at unit speed, and so arrives at the new
site at time $t+1$, at which time it starts to serve the new queue.
Of interest is the asymptotic behaviour of $S(t)$, the location of the server at time $t \geq 0$.

There are 3 cases. The least interesting case is when $\lambda > \mu$. In this case, any queue under service is transient, so each time that the server starts serving a queue, there
 is (uniformly) positive probability that the server remains at the site for all time.
 Thus with probability 1, the server changes site only finitely many times,
 eventually remaining at one site for ever, so $S(t)$ converges almost surely (a.s.) See Theorem~1.1 of~\cite{km}.

The main object of study in~\cite{km} is the case $\lambda < \mu$.
Theorem 1.2 of \cite{km} shows that in this case the server changes its
 direction only finitely many times, so that the server
 eventually moves from site to site in a single direction, and $S(t) \to \pm \infty$ a.s.
The intuition behind this result is as follows. Any queue under service is now recurrent, so becomes emptied
 in finite time.
 Consider the server's first visit to site $x >0$ at time $t$ (say).
 At time $t$, with very high probability, the queue at $x-1$ will be essentially empty,
 while the queues at $x$ and $x+1$ will have lengths concentrated about $\lambda t$.
It takes time  about 
$(\lambda/\mu) t$ for the server
 to empty the queue at site $x$, and in this time there will be about $(\lambda^2/\mu) t$
 new arrivals at the queues at sites $x-1$ and $x+1$.
 So by the time the queue at site $x$ is emptied,
 the queue lengths at sites $x-1$ and $x+1$
 are about $(\lambda^2/\mu) t$ and $(\lambda + (\lambda^2/\mu)) t$
 respectively. The fluctuations are of order $t^{1/2}$, and so with very high probability, the server will choose to go to site $x+1$ next. 
 
In the present paper, we study the critical case $\lambda = \mu$, which was left largely open in~\cite{km}.
It is clear that this case is rather more delicate. Again, any queue under service is   recurrent.
 But now an attempt to follow the idea of the argument sketched
 in the previous case reveals a new issue.
 Once again, upon the server's first arrival at site $x$,
 the queues at $x$ and $x+1$ will have lengths
  about $\lambda t$ while the queue at site $x-1$ will be essentially empty.
  But now the queue at site $x$ is critically recurrent,
  and so typically takes time of order $t^2$ to empty.
  In this time, the fluctuations in the new arrivals at sites
  $x -1$ and $x+1$ are of order $t$, i.e., on the same scale
  as the initial difference in queue lengths. So it seems likely that the server will change direction many times;
an understanding of the details of this behaviour seems necessary to obtain the asymptotic
behaviour of the server.

In the case $\lambda = \mu$, Menshikov and Kurkova~\cite{km} proved that
 \begin{equation}
\label{eq:limsup}   \limsup_{t \to \infty} | S(t) | = + \infty , \as, \end{equation}
showing that the server does not get stuck.
Our main result is that the server is \emph{recurrent}, in the sense
that it returns to every site infinitely many times:

\begin{theorem}
\label{thm:recurrent}
Suppose that $\lambda = \mu \in (0,\infty)$. Then, a.s., for every $x \in \R$, the set
$\{ t \geq 0 : S(t) = x \}$ is unbounded.
\end{theorem}

We also establish the following result
on the 
growth rate of $S(t)$.

\begin{theorem}
\label{thm:growth}
Suppose that $\lambda = \mu \in (0,\infty)$. Then, a.s.,
\begin{align*}
\limsup_{t \to \infty} \frac{S(t)}{\sqrt{\log \log t \log \log \log \log t}} & =   \sqrt{\frac{6}{\log 2}}, \text{ and}\\
\liminf_{t \to \infty} \frac{S(t)}{\sqrt{\log \log t \log \log \log \log t}} & = - \sqrt{\frac{6}{\log 2}}.
\end{align*}
\end{theorem}
 
The rest of this paper is devoted to the proofs of Theorems~\ref{thm:recurrent} and~\ref{thm:growth}.
We give a concrete construction of the process, via a discrete-time process that is the basic object of study in this paper,
in Section~\ref{sec:discrete}. In Section~\ref{sec:discrete} we also describe the main steps in the proof and give the outline of the rest of the paper.
 
The greedy server on $\Z$ is a variant of the \emph{greedy server} problem
introduced in~\cite{cg} and surveyed in~\cite{rnfk},
in which a server greedily moves from 
job to job arriving randomly in some space, such as on a line or a circle. 
Also related is the so-called \emph{greedy walk} problem~\cite{bfl}.
These models have received significant attention over several decades,
in part because the dynamics of the server possess features of both
self-interacting processes and processes in random environments, which remain
very active topics of current research, and because
the problems they pose are challenging. A number of open problems remain: see e.g.~\cite{rnfk,bfl}.

A contiuum analogue
of our problem on $\Z$ is the greedy walk on $\R$, for which it was recently shown that the server
escapes to infinity~\cite{frs}, in contrast to our Theorem~\ref{thm:recurrent}.

\section{Discrete-time process and paper outline}
\label{sec:discrete}

For the remainder of the paper we fix $\lambda = \mu \in (0,\infty)$. We write $\ZP := \{0,1,2,\ldots\}$
and $\N := \{1,2,3,\ldots\}$.

We will study the continuous-time process described in Section~\ref{sec:intro}
via a discrete-time process obtained by observing the full process at the $n$th time at which the server
empties a queue. Consider a Markov process $\Psi_n = (Q_n , X_n, T_n )$ where
$Q_n = ( Q_n (x) )_{x \in \Z} \in  ( \ZP\cup \{ * \} )^\Z$, $X_n \in \Z$,
and $T_n \in \RP$.
If $Q_n (x) \in \ZP$, that is the number of customers at queue $x \in \Z$; if $Q_n(x) = *$ then this indicates that
the queue at $x \in \Z$ has yet to be inspected by the server. The coordinate $X_n$ represents the location of the server when
a queue is emptied for an $n$th time, and $T_n$ represents the total time that has elapsed (i.e.~the sum of
all the services times plus the travel times up to this point).

Set $Q_0 (x) := 0$ for $x \in \{-1,0,+1\}$ and $Q_0 (x) := *$ for $| x | \geq 2$; set $X_0 := 0$ and $T_0:= 0$.
We describe the  law of this process by its Markovian transitions. The random ingredients that go into this description are as follows.
Let $\xi_1, \xi_2, \ldots$ be i.i.d.~with $\Pr ( \xi_1 = +1) = \Pr ( \xi_1 = -1) = 1/2$ (these will be the tie-breaking variables).

We write $P(\kappa)$ to denote a Poisson random variable with mean $\kappa \in \RP$;
for a random variable $W$ on $\RP$ we write $P (W)$ to denote a random variable that,
conditional on $W$, has a Poisson distribution with mean $W$.

Let $(Z_t, t \in \RP)$ denote an $M/M/1$ queue with arrival rate $\lambda$ and
service rate $\mu$, with $Z_0 = k \in \ZP$ initial customers; let $\zeta (k)$ denote 
a random variable distributed as $\inf \{ t \geq 0 : Z_t = 0 \}$, the time to empty
for the queue. Similarly, conditional on a random variable $W \in \ZP$, $\zeta (W)$ is distributed as the time to empty 
the queue started from $Z_0 = W$.

Given $(Q_n, X_n, T_n)$, generate $(Q_{n+1}, X_{n+1}, T_{n+1} )$ as follows.
\begin{itemize}
\item
Define
\[ \eta_{n+1} := \begin{cases} +1 \text{ if } Q_n (X_n +1 ) > Q_n (X_n -1) ,\\
-1 \text{ if } Q_n (X_n -1 ) > Q_n (X_n +1 ) .\end{cases} \]
If $Q_n (X_n+1) = Q_n (X_n -1 )$, then take $\eta_{n+1} = \xi_n$.
\item
Let $X_{n+1} = X_n + \eta_{n+1}$.
\item
Let $\tau_{n+1}$ be distributed as $1$ plus $\zeta ( Q_n ( X_{n+1} ) + P ( \lambda ) )$.
\item
Let $T_{n+1} = T_n  + \tau_{n+1}$.
\item
For every $k \in \Z \setminus \{ X_{n+1} \}$ such that $Q_n (k) \neq *$, take $Q_{n+1} (k)$ to be 
distributed as $Q_n(k) + P (\lambda \tau_{n+1})$, independently for each $k$.
\item
If $Q_n (X_{n+1} +1 ) = *$ then let $Q_{n+1} (X_{n+1} +1)$ be distributed as $P (\lambda T_{n+1} )$, and if
 $Q_n (X_{n+1} -1 ) = *$ then let $Q_{n+1} (X_{n+1} -1)$ be distributed as $P (\lambda T_{n+1} )$ (independently).
\item
Set $Q_{n+1} (X_{n+1} ) = 0$.
\end{itemize}
Note that $Q_n (X_n) =0$ for all $n$, and $T_n = \sum_{k=1}^n \tau_k$.
Let $\cF_n := \sigma(\Psi_0, \xi_0, \Psi_1, \xi_1, \ldots, \Psi_n, \xi_n)$. Clearly both
$\tau_n = T_n - T_{n-1}$ and $\eta_n = X_n - X_{n-1}$ are $\cF_n$-measurable,
but it is important to observe that $X_{n+1}$ and $\eta_{n+1}$ are also $\cF_n$-measurable.
Thus  $(X_n, \eta_n)$ is $\cF_{n-1}$ measurable. Note that $\tau_n \geq 1$, a.s., so that $T_n \geq n$, a.s.

Let $N_t := \max \{ n \in \ZP : T_n \leq t \}$ denote the number of times that a queue
has been emptied by time $t$. Since $T_n \to \infty$ a.s., we have $N_t < \infty$
a.s.~for all $t \in \RP$;
indeed, $N_t \leq T_{N_t} \leq t$, a.s. 
Moreover, $N_t$ is nondecreasing in $t$. Note $T_{N_t +1} > t$.
Thus if $N_t \to N < \infty$, we have $T_{N+1} = \infty$ which contradicts $T_n < \infty$ for all $n$;
hence $N_t \to \infty$ a.s.~as $t \to \infty$.

Observe that $X_{N_t}$ is the most recent queue that was emptied  
prior to time $t$, and $X_{N_t +1}$ is the next queue to be emptied after time $t$.
Also, $T_{N_t} \leq t$ is the time at which the most recently emptied queue was emptied.
It follows that we have the representation
\begin{equation}
\label{eq:S}
S (t) = \begin{cases} X_{N_t +1} & \text{if } t - T_{N_t} \geq 1 , \\
X_{N_t} + (t-T_{N_t}) (X_{N_t +1} - X_{N_t} ) & \text{if } t - T_{N_t} \leq 1 .\end{cases}
\end{equation}

We end this section by outlining the main steps in the proofs of Theorems~\ref{thm:recurrent}
and~\ref{thm:growth}, and some of the underlying intuition.
The first key ingredient is that $\tau_n$ and $T_n$ exhibit \emph{doubly exponential} growth (see Propositions~\ref{prop:lower-bound} and~\ref{prop:upper-bound}).
This very rapid growth in the time-scales suggests an effective memorylessness for the system:
the configuration of the discrete-time system more than one or two time-steps ago is not important.
This provides the intuition behind the second key ingredient in proving Theorem~\ref{thm:recurrent}, which 
is establishing (in Proposition~\ref{prop:turning})
that the \emph{turning probability} converges: remarkably,
\begin{equation}
\label{q-limit} \Pr ( \eta_{n+1} \neq \eta_n \mid \cF_{n-1} ) \to \frac{1}{4} , \as \end{equation}
 Thus the server's motion is asymptotically
similar to the Gillis--Domb--Fisher \emph{correlated random walk}~\cite{cr}. 
In fact, more than convergence in~\eqref{q-limit} is necessary: we need a bound on the \emph{rate of convergence} with $n$.
The double-exponential growth of the time-scales means that fairly rough estimates are enough.
The double-exponential growth is also the origin of the iterated logarithm in Theorem~\ref{thm:growth};
the precise
value of the constant comes in part from the precise value of~\eqref{q-limit}.
A technical device central to the proofs of both theorems is the construction of a function $f(X_n, \eta_n)$ of the process
that is close to a martingale.

The outline of the rest of the paper is as follows.
Section~\ref{sec:queue} collects some results
on the random variables $\zeta(k)$ that we will need in our analysis.
Section~\ref{sec:time} contains the key estimates on the growth of $\tau_n$ and $T_n$.
Section~\ref{sec:turning} contains the convergence result for the turning probability.
Section~\ref{sec:proofs} contains the martingale construction that allows
us to complete the proofs of Theorems~\ref{thm:recurrent}
and~\ref{thm:growth}.
Finally, Appendix~\ref{sec:aux} collects a couple of auxiliary results used in the analysis.

\section{The critically-loaded queue}
\label{sec:queue}

Let $(Z_t, t \in \RP)$ be a continuous-time symmetric
simple random walk on $\Z$ with jump rate $2\lambda$, i.e., for any $x \in \Z$,
for all $t \in \RP$,
\[ \Pr ( Z_{t+h} = x \pm 1 \mid Z_t = x ) = \lambda h + o(h) ,\]
as $h \to 0$.
Suppose that $Z_0 = k \in \ZP$ and let $\zeta (k) := \inf \{ t \geq 0 : Z_t = 0 \}$, the time to reach $0$
started from $k$. Note that up until $\zeta(k)$, $Z_t$ is distributed as the number of customers in an $M/M/1$ queue with arrival rate and service rate both equal to $\lambda \in (0,\infty)$;
so $\zeta(k)$ is the time to empty such an $M/M/1$ queue, started from $k \in \ZP$ initial customers,
as described in Section~\ref{sec:discrete}.

First we collect several straightforward results about $\zeta(k)$ that we need in the rest of the paper.

\begin{lemma}
\label{lem:hittingtime}
\begin{itemize}
\item[(a)] For $k \geq \ell \geq 0$, $\zeta(k)$ stochastically dominates $\zeta(\ell)$.
\item[(b)] For any $\alpha \in (0,2)$ and $c \in (0,\infty)$ there exist $\eps >0$ and $k_1 \in \N$ such that
\[ \Pr ( \zeta(k) \leq c k^\alpha ) \leq \exp \{ -k^\eps \} , \text{ for all } k \geq k_1.\]
\item[(c)] For any $\beta \in (2, \infty)$ and $c \in (0,\infty)$ there exist $\eps >0$ and $k_2 \in \N$ such that
\[ \Pr ( \zeta(k) \geq c k^\beta ) \leq k^{-\eps } , \text{ for all } k \geq k_2.\]
\end{itemize}
\end{lemma}

 Before proving this lemma, we make some observations.
Suppose $Z_0 = k \in \ZP$. Then we can represent $\zeta(k)$ as
\begin{equation}
\label{eq:zeta-rep}
 \zeta(k) = Y_1 + \cdots + Y_k ,\end{equation}
where 
\[ Y_j = \inf \{ t \geq 0 : Z_t = j- 1\}  - \inf \{ t \geq 0 : Z_t = j \} .\]
Note that, by the strong Markov property and the spatial homogeneity of the random walk,
 the $Y_j$ in~\eqref{eq:zeta-rep} are i.i.d.~copies of $\zeta (1)$.

It is well known that for $k \in \N$,  $\zeta (k)$ has density
\[ f_k (u) := \frac{k}{u} I_k (2 \lambda u) \re^{-2 \lambda u}, \text{ for } u >0 ;\]
see for example Sections II.7 and XIV.6 of \cite{feller2}.
Here $I_k$ is the modified Bessel function of the first kind:
\[ I_k (u) := \sum_{j=0}^\infty \frac{ (u/2)^{2j + k}}{j ! (j+k)! } .\]
In particular, the density of $\zeta(1)$ is
\begin{equation}
\label{f-density}
 f (u) := f_1 (u) = \frac{1}{u} I_1 ( 2 \lambda u) \re^{-2\lambda u} = \frac{1}{2\sqrt{\pi\lambda}} u^{-3/2} + O ( u^{-5/2} )  ,
\end{equation}
as $u \to \infty$, using the asymptotic expressions of \cite[p.~203]{watson}. Note that $f(u) >0$ for all $u \in (0,\infty)$.
Let $F (u) := \Pr ( \zeta (1) \leq u)$ and $\bar F (u) := \Pr ( \zeta (1) > u)$. Then, by~\eqref{f-density}, we have
\begin{equation}
\label{eq:G-tail}
 \bar F (u) = \int_u^\infty f (v) \ud v = \frac{1}{\sqrt{ \pi \lambda} } u^{-1/2} + O ( u^{-3/2} ) ,\end{equation}
as $u \to \infty$.

\begin{proof}[Proof of Lemma~\ref{lem:hittingtime}.]
For $0 \leq \ell \leq k$, the representation~\eqref{eq:zeta-rep}  gives $\zeta(k) \geq \zeta (\ell)$ in the obvious
coupling, so we get part~(a).
 
For part (b), we use~\eqref{eq:zeta-rep} and the fact that $Y_j \geq 0$ to write
\begin{align*}
\Pr ( \zeta(k) \leq r ) & \leq \Pr \left( \max_{1 \leq  j \leq k} Y_j \leq r \right) \\
& = \Pr \left( \bigcap_{j=1}^k \{ Y_j \leq r \} \right) \\
& =   \left( 1 - (\pi \lambda (1 + o(1)) r)^{-1/2} \right)^k ,
\end{align*}
as $r \to \infty$, by~\eqref{eq:G-tail}.
In particular, taking $r = c k^\alpha$ with $c, \alpha >0$ gives
\begin{align*} \Pr ( \zeta(k) \leq c k^\alpha ) & \leq \left( 1 - (\pi \lambda c (1 + o(1)) k^\alpha )^{-1/2} \right)^k , \end{align*}
as $k \to \infty$, so that
\[ \log \Pr ( \zeta(k) \leq c k^\alpha )  \leq k \log \left( 1 - (\pi \lambda c (1 + o(1)) k^\alpha )^{-1/2} \right)
\sim - ( \pi \lambda c )^{-1/2} k^{1-(\alpha/2)} ,\]
which gives part~(b).

For part~(c), let $p \in ( 1/\beta , 1/2)$. Note from~\eqref{eq:G-tail} that the $Y_j$ appearing in~\eqref{eq:zeta-rep} have $\Exp ( Y_j^p ) \leq C$
for some $C < \infty$.
Then by subadditivity of the function $y \mapsto y^{p}$ we have from~\eqref{eq:zeta-rep} that $\Exp ( \zeta(k)^p ) \leq \sum_{j=1}^k \Exp (Y_j^p ) \leq Ck$. Hence, by Markov's inequality,
\[ \Pr ( \zeta(k) \geq c k^\beta ) = \Pr ( \zeta(k)^p \geq c^p k^{\beta p} ) \leq Cc^{-p}k^{1-\beta p} ,\]
which, by choice of $p$, gives part~(c).
\end{proof}

Let $\Phi$ be the distribution function of the standard normal distribution,
and let $\bar \Phi (u) := 1 - \Phi (u)$ for $u \in \R$.
We say that $S$ has a \emph{L\'evy distribution} with location parameter $0$ and scale parameter $1$
(see~\cite[\S 1.1]{nolan}) if
$S \in \RP$ has distribution function given  by
\begin{equation} 
\label{levy-cdf}
F_S (u) := \Pr ( S \leq u ) = 2 \bar \Phi ( u^{-1/2} ) , \text{ for } u > 0.\end{equation}
Note that the density $f_S (u) := F_S' (u)$ corresponding to~\eqref{levy-cdf} is
\begin{equation} 
\label{levy-pdf}
f_S (u) = \frac{1}{\sqrt{2\pi}} u^{-3/2} \re^{-u^{-1}/2} , \text{ for } u >0.
\end{equation}

\begin{lemma}
\label{lem:levy}
Let $S$ be a random variable with the distribution given by~\eqref{levy-cdf}.
There exists a constant $C \in \RP$ such that, for all $k \in \N$, 
\[ \sup_{u \in \RP} \left| \Pr ( k^{-2} \zeta(k) \leq u ) - F_S (2 \lambda u)  \right| \leq Ck^{-1} . \]
\end{lemma}
\begin{proof}
For the purposes of this proof only, we take $Z_0 = 0$.
Let $D := D ( \RP, \R)$ denote the space of functions from $\RP \to \R$
that are right-continuous and have left limits, endowed with the Skorokhod metric.
Define for $m \in \N$,
\[ z_m (t) := m^{-1/2} Z_{mt} , \text{ for } t \geq 0.\]
Then $z_m \in D$ for each $m \in \N$.
Let $(b(t), t \in \RP)$ denote standard Brownian motion started at $b_0 = 0$.
The invariance principle for continuous-time random walks implies that as $m \to \infty$,
$z_m \Rightarrow \sqrt{2 \lambda} b$ in the sense of weak convergence
on $D$  (one may apply Theorem~7.1.4 of~\cite[pp.~339--340]{ek}, for example).

For $z \in D$, let $\sigma(z) := \inf \{ t \geq 0 : z(t) > 1 \}$.
For Brownian motion, we have
that
$\sigma (b) = \inf \{ t \geq 0 : b(t) =1 \}$ a.s., and, 
 for any $\eps >0$, $\sup_{0 \leq s \leq \sigma(b) - \eps} b(s) < 1$, a.s.
Thus the set of discontinuities of the mapping $z \mapsto \sigma(z)$ has measure zero
under the measure induced by  Brownian motion (see Section~5.7.5 of \cite{whitt}).
So by the  mapping theorem (see e.g.~Theorem~2.7 of~\cite{bill}) we get $\sigma ( (2 \lambda)^{-1/2} z_m ) \tod S := \sigma (b)$ (here $\tod$ denotes convergence in distribution).
Here
\begin{align*} \sigma ( (2 \lambda)^{-1/2} z_m ) 
& = \inf \{ t \geq 0 : Z_{m t} > \sqrt{ 2 \lambda m } \} \\
& = m^{-1} \inf \{ s \geq 0 : Z_s \geq 1 + \lfloor \sqrt{ 2 \lambda m } \rfloor  \}\\
& \eqd m^{-1} \zeta ( 1 + \lfloor \sqrt{ 2 \lambda m } \rfloor ) ,\end{align*}
where $\eqd$ denotes equality in distribution.
Setting $k = 1 + \lfloor \sqrt{ 2 \lambda m } \rfloor \in \N$ we 
have that
$2\lambda / k^2 = m^{-1} + O(m^{-3/2})$
so that, as $k \to \infty$,
\begin{equation}
\label{S-conv}
 \frac{2\lambda}{k^2} \zeta ( k ) \tod S .\end{equation}
The reflection principle for Brownian motion (see e.g.~\cite[p.~372]{durrett}) shows that
\[ \Pr (S > u ) = 1 - 2 \Pr ( b(u) \geq 1 ) = 2 \Phi ( u^{-1/2} ) -1 ,\]
so $S$ has the distribution given by~\eqref{levy-cdf}.

It remains to estimate the rate of convergence in~\eqref{S-conv}. By~\eqref{eq:G-tail} and Theorem~2.6.7 of~\cite{il}, we have that $\zeta(1)$
is in the normal domain of attraction of a positive stable law with index $1/2$. Indeed, $S$ is stable with index $1/2$ since, by the scaling and strong Markov properties
of Brownian motion, for any $m \in \N$,
\[ S \eqd \inf \{ t \geq 0 : m^{-1} b ( m^2 t ) = 1\} = m^{-2} \inf \{ t \geq 0 : b ( t) = m\} \eqd m^{-2} ( S_1 + \cdots + S_m ) ,\]
where the $S_j$ are independent copies of $S$.
Thus we can apply results on the rate of convergence in the stable central limit theorem for the sum in~\eqref{eq:zeta-rep}.
First note that, by Taylor's theorem,
\begin{align*} \Phi ( u^{-1/2} ) & = \Phi (0) + \Phi' (0) u^{-1/2} + \frac{\Phi'' (0)}{2} u^{-1} + O ( u^{-3/2} ) \\
& = \frac{1}{2} + \frac{1}{\sqrt{2\pi}} u^{-1/2} + O ( u^{-3/2} ),\end{align*}
as $u \to\infty$.
Thus if $\bar F_S (u) := 1 - F_S (u)$ we have from~\eqref{levy-cdf} that
\begin{equation}
\label{S-tail}
 \bar F_S ( 2 \lambda u ) = 2  \Phi ( (2\lambda u)^{-1/2} ) -1 =  \frac{1}{\sqrt{\pi\lambda}} u^{-1/2} + O (u^{-3/2}) ,\end{equation}
as $ u \to \infty$. Combining~\eqref{S-tail} with~\eqref{eq:G-tail} we have that
\[ | F ( u) - F_S ( 2\lambda u ) | = O ( u^{-3/2} ) .\]
This condition enables one to verify standard `pseudomoments' conditions
for Berry--Esseen bounds in stable limit theorems. Indeed, setting $H(u ) = F(u) - F_S (2 \lambda u)$ and
\[ \mu_\ell = \int_0^\infty u^\ell \ud H (u) , \text{ and } \nu_\ell = \int_0^\infty | u | ^\ell | \ud H (u) | ,\]
we have that $\nu_1 < \infty$ and $\mu_0 = 0$, so we may apply the results of~\cite{saty} 
(which has a statement but no proof), \cite{paul} (Corollary~1)
or~\cite{cw} (combine Theorem~3.11 of~\cite[p.~66]{cw} with Lemma~2.5 of~\cite[p.~27]{cw}).
This gives the result.
\end{proof}

\section{Time-scale estimates}
\label{sec:time}

In this section we study the asymptotics of $\tau_n$ and $T_n$. First we have a lower bound.

\begin{proposition}
\label{prop:lower-bound}
For any $\alpha \in (1,2)$, $T_n \geq \tau_n \geq \re^{\alpha^n}$ for all but finitely many $n$, a.s.
\end{proposition}

We also have the following upper bound.

\begin{proposition}
\label{prop:upper-bound}
For any $\beta \in (2,\infty)$, $\tau_n \leq T_n \leq \re^{\beta^n}$ for all but finitely many $n$, a.s.
\end{proposition}

\begin{remark*}
A rough calculation (cf.~Lemma~\ref{lem:stable} below) suggests that in fact we may have
\[ \lim_{n \to \infty} \frac{ \log \tau_n}{2^n} = \gamma , \as, \]
and the same for $T_n$. Here $\gamma \in (0,\infty)$ is a random variable with representation $\gamma = \sum_{i=1}^{\infty} 2^{-i} \log (\lambda S_{i}/2)$,
where $S_1, S_2, \ldots$ are independent random variables with distribution given by~\eqref{levy-cdf}.
To establish Theorems~\ref{thm:recurrent} and~\ref{thm:growth} however, the bounds in Propositions~\ref{prop:lower-bound}
and~\ref{prop:upper-bound} are sufficient (in fact, for Theorem~\ref{thm:recurrent}, we need only the lower bound).
\end{remark*}

We work towards the proof of Proposition~\ref{prop:lower-bound}.
We start with a crude bound. Here and elsewhere, `i.o.' stands for `infinitlely often'.

\begin{lemma}
\label{lem:sqrt-tail}
We have $\tau_n \geq   n^2$ i.o., a.s.
\end{lemma}
\begin{proof}
 Given $\cF_{n+1}$, we have from the description in Section~\ref{sec:discrete} and  Lemma~\ref{lem:hittingtime}(a) that $\tau_{n+2}$ stochastically dominates
$\zeta ( Q_{n+1} )$, where $\zeta( Q_{n+1} )$ depends on $\cF_{n+1}$ only through $Q_{n+1} := Q_{n+1} (X_{n+2})$.
Moreover, 
since $X_{n+2} \neq X_{n+1}$, we have that the queue at $X_{n+2}$ is not being 
served
between times $T_n$ and $T_{n+1}$, and in that time accumulates
a Poisson number of arrivals with mean $\lambda \tau_{n+1} \geq \lambda$, since $\tau_{n+1} \geq 1$ a.s.
Hence, given $\cF_n$, 
$Q_{n+1}$ stochastically dominates  
a Poisson random variable with mean $\lambda$. Thus we get
\begin{align*} \Pr ( \tau_{n+2} > r \mid \cF_n) 
& \geq \Exp ( \1 { Q_{n+1} \geq 1 } \Pr ( \tau_{n+2} > r \mid \cF_{n+1} ) \mid \cF_n) \\
& \geq \Exp ( \1 { Q_{n+1} \geq 1 } \Pr ( \zeta ( Q_{n+1} ) > r \mid \cF_{n+1} ) \mid \cF_n) \\
& \geq \Pr ( Q_{n+1} \geq 1 \mid \cF_n) \Pr ( \zeta(1) > r ) \\
& \geq \Pr ( P (\lambda ) \geq 1 ) \Pr ( \zeta (1 ) > r ) \\
& \geq (1 - \re^{-\lambda} ) ( \pi \lambda ( 1 + o(1) ) r )^{-1/2} ,
\end{align*}
by~\eqref{eq:G-tail}, as $r \to \infty$, uniformly in $n$. It follows that
there exist $c > 0$ and $r_0 \geq 1$ such that for any $n \in \ZP$ and any $r \geq r_0$,
\begin{equation}
\label{eq:sqrt-tail}
 \Pr ( \tau_{n+2} > r \mid \cF_{n}) \geq c r^{-1/2} , \as \end{equation}

 Let $A_n = \{ \tau_n > n^2 \}$, $B_n = A_{2n}$,
and $\cG_n = \cF_{2n}$.
Now 
taking $n_0 \in \N$ large enough so that $(n+2)^2 \geq r_0$ for all $n \geq n_0$, we have from~\eqref{eq:sqrt-tail} that
$\Pr ( A_{n+2} \mid \cF_n ) \geq \frac{c}{n+2}$, a.s.,  for all $n \geq n_0$.
Hence $B_n \in \cG_n$ and 
\[ \sum_{n \geq n_0} \Pr ( B_{n+1} \mid \cG_{n}) =  \sum_{n \geq n_0} \Pr (A_{2n+2} \mid \cF_{2n})
\geq \sum_{n \geq n_0}  \frac{c}{2n+2} = \infty, \as \]
Thus L\'evy's extension of the Borel--Cantelli lemma  (see e.g.~\cite[Corollary~7.20]{kall}) implies that
$B_n$ occurs infinitely often, and hence $A_n$ occurs infinitely often.
\end{proof}
 
The next result gives conditions under which an a.s.~lower bound for $\tau_n$ that holds
infinitely often can be converted into a bound that holds all but finitely often.

\begin{lemma}
\label{lem:io-to-all-but-fo}
Let $b_1, b_2, \ldots \in (0,\infty)$  and $\alpha \in (1,2)$ be such that
\begin{itemize}
\item[(i)] $\lim_{n \to \infty} b_n = \infty$; 
\item[(ii)] $\lim_{n\to\infty} (  b_{n+1} / b_n^{\alpha}) = 0$; 
\item[(iii)] $\sum_{n=1}^\infty \re^{-b_n^{\eps}} < \infty$ for any $\eps >0$.
\end{itemize}
Suppose that $\Pr ( \tau_n > b_n \io) = 1$. Then $\tau_n > b_n$ for all but finitely many $n$, a.s.
\end{lemma}
\begin{proof}
We have from Lemma~\ref{lem:hittingtime}(a) that, given $\cF_n$,
 $\tau_{n+1}$ stochastically dominates
$\zeta ( Q_{n} )$, where $Q_{n} := Q_{n} (X_{n+1})$. Thus
\begin{align*}
\Pr ( \tau_{n+1} \leq b_{n+1} , \tau_n > b_n ) & \leq \Pr ( \zeta (Q_n) \leq b_{n+1}, \tau_n > b_n ) \\
& \!\!\!\!\!\!\!\!\!\!\!\!{} \leq \Pr ( \zeta (Q_n) \leq b_{n+1}, Q_n \geq \lambda \tau_n /2 , \tau_n > b_n ) 
+\Pr ( Q_n \leq \lambda \tau_n /2 , \tau_n > b_n ) .\end{align*}
Moreover, given $\tau_n$, $Q_{n}$ stochastically dominates $P ( \lambda \tau_n)$, so that
\begin{align*}
\Pr ( \tau_{n+1} \leq b_{n+1} , \tau_n > b_n )  
& \leq \Pr ( \zeta (Q_n) \leq b_{n+1}, Q_n \geq \lambda b_n /2 ) + \Pr ( P ( \lambda \tau_n ) \leq \lambda \tau_n /2 , \tau_n > b_n ) .
\end{align*}
By another application of Lemma~\ref{lem:hittingtime}(a), we have 
\begin{align*} \Pr ( \zeta (Q_n) \leq b_{n+1}, Q_n \geq \lambda b_n /2 ) 
& \leq \Pr ( \zeta ( \lfloor \lambda b_n /2 \rfloor ) \leq b_{n+1} ) \\
& \leq \Pr ( \zeta ( \lfloor \lambda b_n / 2 \rfloor ) \leq  \lfloor \lambda b_n / 2 \rfloor^\alpha ) ,\end{align*}
for all $n$ sufficiently large, since (i) and~(ii) imply that $b_{n+1} < \lfloor \lambda b_n /2 \rfloor^\alpha$ for $n$ large enough.
Hence, by Lemma~\ref{lem:hittingtime}(b) and the fact that $b_n \to \infty$, we have that for some $\eps>0$,
\[ \Pr ( \zeta (Q_n) \leq b_{n+1}, Q_n \geq \lambda b_n /2 ) \leq \re^{-b_n^\eps} ,\]
for all $n$ sufficiently large. On the other hand,
\[ \Pr ( P ( \lambda \tau_n ) \leq \lambda \tau_n /2 , \tau_n > b_n ) \leq \sup_{s \geq b_n} \Pr ( P ( \lambda s) \leq \lambda s /2 ) \leq \re^{-\delta b_n} ,\]
for some $\delta >0$, by standard Poisson tail bounds (see e.g.~\cite[p.~17]{penrose}).
Combining these estimates and using the fact that $b_n \to \infty$, we get, for some $\eps>0$,
\[ \Pr ( \tau_{n+1} \leq b_{n+1} , \tau_n > b_n )  \leq \re^{-b_n^\eps} ,\]
for all $n$ sufficiently large.
Then by (iii) and the Borel--Cantelli lemma, we have that $\{ \tau_{n+1} \leq b_{n+1} , \tau_n > b_n \}$
occurs only finitely often, a.s. In other words, for all $n$ sufficiently large we have $\tau_n > b_n$ implies $\tau_{n+1} > b_{n+1}$,
and since $\tau_n > b_n$ i.o., the result follows.
\end{proof}

We can now deduce the following lower bound, which, despite being far from best possible,
is an important step in proving Proposition~\ref{prop:lower-bound}.

\begin{corollary}
\label{cor:tau-nsq}
Almost surely, for all but finitely many $n$, $\tau_n \geq n^2$.
\end{corollary}
\begin{proof}
Taking $b_n = n^2$ in Lemma~\ref{lem:io-to-all-but-fo}, and applying Lemma~\ref{lem:sqrt-tail},
we obtain the result.
\end{proof}

The next result, showing that queues are rarely much shorter than we would expect,
will be used a couple of times.

\begin{lemma}
\label{lem:Q-small}
Almost surely, for all but finitely many $n$, $Q_n (X_{n+1} ) > \lambda \tau_n - \tau_n^{3/4}$.
\end{lemma}
\begin{proof}
Let $Q_n := Q_n ( X_{n+1} )$. Let $n \in \N$. Then, given $\cF_{n-1}$,
 $Q_{n}$ stochastically dominates a Poisson random variable with mean $\lambda \tau_n$.
Thus
\begin{align*}
\Pr ( Q_n \leq \lambda \tau_n - \tau_n^{3/4} \mid \cF_{n-1} ) & \leq \Pr ( P (\lambda \tau_n) \leq \lambda \tau_n - \tau_n^{3/4} , \tau_n \geq n \mid \cF_{n-1} )
+ \Pr ( \tau_n \leq n \mid \cF_{n-1} ) \\
& \leq \sup_{s \geq n} \Pr ( P (\lambda s) \leq \lambda s - s^{3/4} ) + \Pr ( \tau_n \leq n \mid \cF_{n-1} ) \\
& \leq \re^{-n^\eps} + \Pr ( \tau_n \leq n \mid \cF_{n-1} ) ,
\end{align*}
for some $\eps >0$ and all $n$ sufficiently large, by Poisson concentration (see e.g.~\cite[p.~17]{penrose}).
In particular, since by Corollary~\ref{cor:tau-nsq}, $\tau_n \leq n$ only finitely often, a.s., L\'evy's extension of the Borel--Cantelli lemma (see e.g.~\cite[Corollary~7.20]{kall}) implies that 
$\sum_{n \geq 1} \Pr ( \tau_n \leq n \mid \cF_{n-1} ) < \infty$, and hence 
\[ \sum_{n \geq 1} \Pr ( Q_n \leq \lambda \tau_n - \tau_n^{3/4} \mid \cF_{n-1} ) < \infty, \as,\]
which gives the result.
\end{proof}

The next result gives the final ingredient in 
 the proof of Proposition~\ref{prop:lower-bound}, and a bound that we will use later.

\begin{lemma}
\label{lem:tau-growth}
Let $\alpha \in (1,2)$. Then for some $\eps >0$,
\[ \Pr ( \tau_{n+1} \leq \tau_n^\alpha \mid \cF_n ) \leq \re^{-n^\eps}, \text{ for all but finitely many } n, \as \]
In particular, a.s., $\tau_{n+1} \geq \tau_n^\alpha$ for all but finitely many $n$,
and $\tau_n /\tau_{n+1} \to 0$, a.s.
\end{lemma}
\begin{proof}
Let $Q_n := Q_n (X_{n+1})$.
Let $\alpha \in (1,2)$. Given $\cF_{n}$, we have from Lemma~\ref{lem:hittingtime}(a) that $\tau_{n+1}$ stochastically dominates
$\zeta ( Q_{n} )$. Hence
\begin{align*}
\Pr ( \tau_{n+1} \leq \tau_n^\alpha \mid \cF_n) & \leq \Pr ( Q_n \leq \lambda \tau_n /2 \mid \cF_n ) + \Pr ( Q_n > \lambda \tau_n /2, \zeta(Q_n) \leq  \tau_n^\alpha \mid \cF_n ) .\end{align*}
Here by Lemma~\ref{lem:hittingtime}(a) once more, we have
\begin{align*}
\Pr ( Q_n > \lambda \tau_n /2, \zeta(Q_n) \leq \tau_n^\alpha \mid \cF_n ) 
& \leq \Pr ( \zeta ( \lfloor \lambda \tau_n /2 \rfloor ) \leq  \tau_n^\alpha \mid \cF_n ) \\
& \leq \exp \{ -  \lfloor \lambda \tau_n/2 \rfloor^\eps \} ,
\end{align*}
for some $\eps>0$ and all $n$ sufficiently large,
by Lemma~\ref{lem:hittingtime}(b).
By Corollary~\ref{cor:tau-nsq} we have $\tau_n \geq n^2$ for all $n$ sufficiently large, so since $Q_n$ and $\tau_n$
are $\cF_n$-measurable, we get
\[ \Pr ( \tau_{n+1} \leq \tau_n^\alpha \mid \cF_n) \leq \1 {Q_n \leq \lambda \tau_n /2} + \re^{-n^\eps} ,\]
for all but finitely many $n$. It follows from Lemma~\ref{lem:Q-small} that the indicator here vanishes, a.s., for all but finitely many $n$.
The probability bound in the lemma follows.

A consequence of the probability bound is
\[ \sum_{n \geq 0} \Pr ( \tau_{n+1} \leq \tau_n^\alpha \mid \cF_n) < \infty, \as \]
Hence, by L\'evy's extension of the Borel--Cantelli lemma, 
$\tau_{n+1} \geq \tau_n^\alpha$ for all but finitely many $n$. Moreover,
 Corollary~\ref{cor:tau-nsq} shows that $\tau_n \to \infty$ and hence, for all but finitely many $n$,
$\tau_n / \tau_{n+1} \leq \tau_n^{1-\alpha} \to 0$.
\end{proof}

We can now complete the proof of Proposition~\ref{prop:lower-bound}.

\begin{proof}[Proof of Proposition~\ref{prop:lower-bound}.]
Lemma~\ref{lem:tau-growth} shows that there is some $N_1$
with $\Pr ( N_1 < \infty ) =1 $ such that 
\begin{equation}
\label{eq:tau-growth}
\tau_{n+1} \geq \tau_n^\alpha \text{ for all } n \geq N_1 .
\end{equation}
Since (by Corollary~\ref{cor:tau-nsq})
$\tau_n \to \infty$, a.s., we have $\tau_N \geq \re$ for some a.s.~finite $N \geq N_1$.
Then iterating~\eqref{eq:tau-growth} we have $\tau_{N + k} \geq \re^{\alpha^k}$ for all $k \geq 0$.
Take $\tilde \alpha \in (1,\alpha)$. Then
\[ \tau_n \geq \re^{\alpha^{n-N} } \1 { n \geq N } \geq \re^{\tilde \alpha^n} ,\]
for all but finitely many $n$, giving the result.
\end{proof}

\begin{remark*}
\emph{A postiori}, armed with Proposition~\ref{prop:lower-bound}, one can greatly
improve the probability bound in Lemma~\ref{lem:tau-growth}; as stated, however, it is adequate for
its use later in the paper.
\end{remark*}

The next result shows that $T_n / \tau_n \to 1$, a.s., and will be useful in the next section as well as in the proof
of Proposition~\ref{prop:upper-bound}.

\begin{lemma}
\label{lem:T-growth}
Almost surely, for all but finitely many $n$, $T_n/\tau_n \leq 1 + \re^{-6n}$.
\end{lemma}
\begin{proof}
Let $\alpha \in (1,2)$.
We have from Lemma~\ref{lem:tau-growth} and Proposition~\ref{prop:lower-bound} that there exists $N$ with $\Pr (N<\infty) =1$
such that $\tau_{n+1} \geq \tau_n^\alpha$ and $\tau_n \geq \re^{\alpha^n}$ for all $n \geq N$.
Set 
\[ K := 1 + \max_{1 \leq m \leq N-1} \prod_{k=m}^{N-1} \frac{\tau_k}{\tau_{k+1}} .\]
Then since $1 \leq \tau_k < \infty$, a.s., we have that $K < \infty$, a.s. Now, for $n > N$,
\begin{align*}
\max_{1 \leq m \leq n-1} \frac{\tau_m}{\tau_n} & = \max_{1 \leq m \leq n-1} \prod_{k=m}^{n-1} \frac{\tau_k}{\tau_{k+1}}   \\
& \leq \max_{1 \leq m \leq n-1} \prod_{k=m}^{N-1} \frac{\tau_k}{\tau_{k+1}} \prod_{k=N}^{n-1} \frac{\tau_k}{\tau_{k+1}} \1 { m \leq N-1}
+ \max_{1 \leq m \leq n-1}  \prod_{k=m}^{n-1} \frac{\tau_k}{\tau_{k+1}} \1 { m \geq N} .
\end{align*}
For $k \geq N$ we have that $\tau_k/\tau_{k+1} \leq \tau_k^{1-\alpha} \leq 1$,
so for any $m \in \{ N, \ldots, n-1\}$,
\[ \prod_{k=N}^{n-1} \frac{\tau_k}{\tau_{k+1}} \leq \prod_{k=m}^{n-1} \frac{\tau_k}{\tau_{k+1}} \leq \frac{\tau_{n-1}}{\tau_n} .\]
Hence, for $n > N$,
\begin{align*}
\max_{1 \leq m \leq n-1} \frac{\tau_m}{\tau_n} 
&
\leq K  \frac{\tau_{n-1}}{\tau_{n}}
\leq K \tau_{n-1}^{1-\alpha}
\leq K 
\re^{-(\alpha-1)\alpha^{n-1}} .\end{align*}
It follows that,  a.s.,
\[ \max_{1 \leq m \leq n-1} \frac{\tau_m}{\tau_n} \leq K \re^{-8n} \leq \re^{-7n},\]
for all but finitely many $n$. Now the result follows from the fact that 
\[ \frac{T_n}{\tau_n} = 1 + \sum_{m=1}^{n-1} \frac{\tau_m}{\tau_n} \leq 1 + n\max_{1 \leq m \leq n-1} \frac{\tau_m}{\tau_n}  . \qedhere \]
\end{proof}

We also need a complementary result to Lemma~\ref{lem:Q-small}.

\begin{lemma}
\label{lem:Q-big}
Almost surely, for all but finitely many $n$, $Q_n (X_{n+1} ) <  \lambda ( 1 + \re^{-5n} ) \tau_n$.
\end{lemma}

\begin{proof}
At time $T_n$, the queue at $X_n$ is emptied and the queues at $X_n \pm 1$ are inspected; let
$Q_L = Q_n (X_n -1)$ and $Q_R = Q_n (X_n + 1)$. Then $Q_n := Q_{n} (X_{n+1} ) = \max \{ Q_L, Q_R \}$.
Suppose that the queue at $X_n -1$ was most recently emptied at some time
$T_L < T_n$, and that the queue at $X_{n} +1$ was most recently emptied
at some time $T_R < T_n$. After the time at which it was most recently emptied, each queue has been inspected
a finite number of times, and,
because the queue was not served at any point after it was last emptied, on each inspection it was found to
be no larger than the queue to which it was being compared. 
Each such inspection therefore
 (see Lemma~\ref{lem:reduced}) stochastically reduces the queue length. Thus
immediately before the inspection at time $T_n$, we have that $Q_L$ is stochastically
dominated by $P ( \lambda (T_n-T_L))$ and $Q_R$ is stochastically dominated by
$P( \lambda (T_n-T_R))$. It follows that $Q_n$ is stochastically
dominated by the maximum of two $P ( \lambda T_n)$ random variables.
Thus
\begin{align*}
 \Pr ( Q_n \geq \lambda T_n + T_n^{3/4} )
& \leq 2 \Pr ( P ( \lambda T_n ) \geq  \lambda T_n + T_n^{3/4} ) .\end{align*}
Now, since $T_n \geq n$ a.s., we get
\begin{align*}
\Pr ( P ( \lambda T_n ) \geq  \lambda T_n + T_n^{3/4} )
& \leq  \sup_{s \geq n} \Pr ( P (\lambda s ) \geq \lambda s + s^{3/4} )  
  \leq \re^{-n^\eps} ,\end{align*}
for some $\eps>0$ and all $n$ sufficiently large,
by standard Poisson tail bounds (see e.g.~\cite[p.~17]{penrose}).
Thus the Borel--Cantelli lemma implies that
$Q_n \leq \lambda T_n + T_n^{3/4}$ for all but finitely many $n$, a.s.
Lemma~\ref{lem:T-growth} shows that $T_n \leq (1 + \re^{-6n} ) \tau_n$
for all but finitely many $n$, so
\[ Q_n \leq \lambda (1 + \re^{-6n} ) \tau_n + 2 \tau_n^{3/4} .\]
Now, by Proposition~\ref{prop:lower-bound}, for all but finitely many $n$, 
\[ 2 \tau_n^{3/4} = 2 \tau_n \cdot \tau_n^{-1/4} \leq \tau_n \cdot \re^{-7n}
\leq \lambda \re^{-6n} \tau_n ,\]
which gives the result.
\end{proof}

Now we can complete the proof of Proposition~\ref{prop:upper-bound}.

\begin{proof}[Proof of Proposition~\ref{prop:upper-bound}.]
Let $\beta > 2$ and set $Q_n := Q_n (X_{n+1})$.
Given $\cF_n$, $\tau_{n+1}$ is distributed as $1+\zeta(Q_n + \nu )$ where $\nu \sim P(\lambda)$.
Thus
\begin{align*}
 \Pr ( \tau_{n+1} > T_n^\beta \mid \cF_n) & = \Pr ( 1 + \zeta(Q_n +\nu)  > T_n^\beta \mid \cF_n) \\
& \leq \Pr (  \zeta (2 Q_n) > T_n^\beta - 1, \nu \leq Q_n \mid \cF_n ) + \Pr ( \nu > Q_n \mid \cF_n ) .\end{align*}
We have by Markov's inequality and Lemma~\ref{lem:Q-small} that
\[ \Pr (\nu > Q_n \mid \cF_n ) \leq \frac{\lambda}{Q_n} \leq \frac{2 }{ \tau_n} \leq \re^{-2n} ,\]
for all but finitely many $n$, a.s., by Proposition~\ref{prop:lower-bound}.
On the other hand,
\begin{align*}  \Pr (  \zeta (2 Q_n) > T_n^\beta - 1, \nu \leq Q_n \mid \cF_n ) & \leq \Pr ( \zeta ( 2 Q_n ) > T_n^\beta -1, Q_n \leq 2 \lambda T_n \mid \cF_n)\\
& {} \qquad {} + \Pr ( Q_n > 2 \lambda T_n \mid \cF_n) \\
& \leq \Pr ( \zeta ( \lceil 4 \lambda T_n \rceil ) > T_n^\beta -1 \mid \cF_n ) + \1{ Q_n > 2 \lambda T_n } \\
& \leq T_n^{-\eps} ,\end{align*}
for some $\eps>0$ and
all but finitely many $n$, a.s., by Lemma~\ref{lem:Q-big}, Lemma~\ref{lem:hittingtime}(c), and the fact that $T_n \geq \tau_n \to \infty$.
Hence, since $T_n \geq \tau_n$, we have from Proposition~\ref{prop:lower-bound} that a.s.,
\[  \Pr ( \tau_{n+1} > T_n^\beta \mid \cF_n) \leq \re^{-n} ,\]
for all but finitely many $n$.
It follows that $\tau_{n+1} \leq T_n^\beta$ for all but finitely many $n$, a.s.
Let $\tilde \beta > \beta$.
Then Lemma~\ref{lem:T-growth} shows that $T_{n+1} \leq 2 T_n^\beta \leq T_n^{\tilde\beta}$, for all $n \geq N$
with $\Pr ( N < \infty ) = 1$.
It follows that
$T_n \leq T_N^{\tilde \beta^n}$ for all $n$.
Since $\tilde \beta > 2$ was arbitrary, the result follows.
\end{proof}

\section{Turning probability}
\label{sec:turning}

For $n \in \N$ define
\begin{equation}
\label{q-def} q_n := \Pr ( \eta_{n+1} \neq \eta_n \mid \cF_{n-1} ). \end{equation}
The main result of this section is the following.

\begin{proposition}
\label{prop:turning}
Let $q := 1/4$. Then there exists $\eps>0$ such that, a.s., for all but finitely many $n$,
$| q_n - q | \leq \re^{-n^\eps}$.
\end{proposition}

We work towards the proof of Proposition~\ref{prop:turning}.
We need the following result.

\begin{lemma}
\label{lem:stable}
Let $S$ be a random variable with the distribution given by~\eqref{levy-cdf}. Then
\[ \sup_{u \in [0, \re^n]} \left| \Pr \Big( \frac{\tau_n}{\tau_{n-1}^2} \leq u \Bigmid \cF_{n-1} \Big) - F_S ( 2 u/\lambda)   \right| \leq \re^{-2n}, \as ,\]
for all but finitely many $n$.
\end{lemma}
\begin{proof}
We have from Lemma~\ref{lem:hittingtime}(a) that, given $\cF_{n-1}$,
 $\tau_{n}$ stochastically dominates
$\zeta ( Q_{n-1} )$, where $Q_{n-1} := Q_{n-1} (X_{n})$.
It follows that for $u \in \RP$,
\begin{align*}
\Pr ( \tau_n \leq u \tau_{n-1}^2 \mid \cF_{n-1} ) & \leq \Pr ( \zeta (Q_{n-1} ) \leq u \tau_{n-1}^2 \mid \cF_{n-1} ) \\
& \leq \Pr ( \zeta (Q_{n-1} ) \leq u \tau_{n-1}^2, Q_{n-1} \geq \lambda \tau_{n-1} - \tau_{n-1}^{3/4} \mid \cF_{n-1} ) \\
& {} \qquad {} + \Pr ( Q_{n-1} \leq \lambda \tau_{n-1} - \tau_{n-1}^{3/4} \mid \cF_{n-1} ) \\
& \leq \Pr ( \zeta ( \lfloor \lambda \tau_{n-1} - \tau_{n-1}^{3/4} \rfloor ) \leq u \tau_{n-1}^2 \mid \cF_{n-1} ) ,\end{align*}
for all but finitely many $n$, a.s., by Lemma~\ref{lem:Q-small}. Write $\gamma_{n-1} :=  \tau_{n-1} - \lambda^{-1} \tau_{n-1}^{3/4}$. 
Then  
\begin{align*}
\Pr ( \zeta ( \lfloor \lambda \gamma_{n-1} \rfloor ) \leq u \tau_{n-1}^2 \mid \cF_{n-1} ) 
& \leq \Pr ( \lfloor \lambda \gamma_{n-1} \rfloor^{-2} \zeta ( \lfloor \lambda \gamma_{n-1} \rfloor ) \leq u  ( \lambda \gamma_{n-1} - 1)^{-2} \tau_{n-1}^2 \mid \cF_{n-1} ) \\
& \leq F_S ( 2 u \lambda ( \lambda \gamma_{n-1} - 1)^{-2} \tau_{n-1}^2 ) + C \gamma_{n-1}^{-1},
\end{align*}
for some $C < \infty$ and $\eps >0$ not depending on $u$ or $n$, by Lemma~\ref{lem:levy}.
By Proposition~\ref{prop:lower-bound} we have $\gamma_{n-1}^{-1} \leq \re^{-3n}$
for all but finitely many $n$, while
\[ | ( \lambda \gamma_{n-1} - 1)^{-2} \tau_{n-1}^2 - \lambda^{-2} | \leq C \tau_{n-1}^{-1/4} \leq \re^{-4n} ,\]
for all but finitely many $n$. Since the density of $S$ as given in~\eqref{levy-pdf} is uniformly bounded,
\[  F_S ( 2 u \lambda ( \lambda \gamma_{n-1} - 1)^{-2} \tau_{n-1}^2 )  \leq F_S ( 2 \lambda^{-1} u) + C u \re^{-4n} ,\]
for all $u$. Combining our estimates gives, for all but finitely many $n$, a.s., 
\begin{equation}
\label{stable-upper}
 \Pr ( \tau_n \leq u \tau_{n-1}^2 \mid \cF_{n-1} ) \leq F_S ( 2 \lambda^{-1} u) + \re^{-2n}, \text{ for all } u \in [0,\re^n],
\end{equation}
which is one half of the required result.

For the corresponding lower bound, we have that, given $\cF_{n-1}$,
$\tau_n$ is distributed as $1 + \zeta( Q_{n-1} + \nu)$ where $\nu \sim  P (\lambda )$. Then,
by Lemma~\ref{lem:hittingtime}(a),
\begin{align*}
\Pr ( \tau_n \leq u \tau_{n-1}^2 \mid \cF_{n-1} ) &
\geq \Pr ( 1 + \zeta ( Q_{n-1} + \lfloor \re^n \rfloor ) \leq u \tau_{n-1}^2 , \nu \leq \re^n \mid \cF_{n-1} )\\
& \geq \Pr ( \zeta ( Q_{n-1} + \lfloor \re^n \rfloor ) \leq u \tau_{n-1}^2 -1 \mid \cF_{n-1} ) - \Pr ( \nu > \re^n \mid \cF_{n-1} ).
\end{align*}
By Markov's inequality, $\Pr ( \nu > \re^n \mid \cF_{n-1} ) \leq \lambda \re^{-n}$. 
Let $E_n = \{ Q_{n} \leq \lambda ( 1 + \re^{-5n} ) \tau_{n} \}$. Then we have that
\begin{align*}
& {} \quad ~ {}\Pr ( \zeta (Q_{n-1} + \lfloor \re^n \rfloor ) \leq  u \tau_{n-1}^2 -1 , E_{n-1} \mid \cF_{n-1} )  \\
& \geq \Pr ( \zeta( \lfloor \lambda (1+\re^{-5(n-1)} ) \tau_{n-1} \rfloor + \lfloor \re^n \rfloor ) \leq  u \tau_{n-1}^2 -1 \mid \cF_{n-1} )  - \2{  E^\rc_{n-1}  } \\
& \geq \Pr ( \zeta( \lfloor \lambda (1+\re^{-5(n-1)} ) \tau_{n-1}   +   \re^n \rfloor ) \leq  u \tau_{n-1}^2 -1 \mid \cF_{n-1} ),
\end{align*}
for all but finitely many $n$, a.s., by Lemma~\ref{lem:Q-big}. Set $\gamma_{n-1} = ( 1+\re^{-5(n-1)} ) \tau_{n-1} + \lambda^{-1} \re^n$.
Then
\begin{align*}
\Pr ( \zeta ( \lfloor \lambda \gamma_{n-1} \rfloor ) \leq  u \tau_{n-1}^2 -1 \mid \cF_{n-1} ) &
\geq \Pr ( \lfloor \lambda \gamma_{n-1}\rfloor^{-2}  \zeta ( \lfloor \lambda \gamma_{n-1} \rfloor ) 
\leq   \lambda^{-2} \gamma_{n-1}^{-2} (  u \tau_{n-1}^2 -1 ) \mid \cF_{n-1} ) \\
& \geq F_S ( 2 \lambda^{-1} \gamma_{n-1}^{-2} (  u \tau_{n-1}^2 -1 ) ) - \re^{-3n} ,
\end{align*}
for all but finitely many $n$, a.s.,
by Lemma~\ref{lem:levy} and Proposition~\ref{prop:lower-bound}, since $\gamma_{n-1} \geq \tau_{n-1}$.
Proposition~\ref{prop:lower-bound} also shows that $\gamma_{n-1} \geq \tau_{n-1} > \re^{2n}$ and
\[ 1 \geq \gamma_{n-1}^{-2} \tau_{n-1}^2  \geq 1 - \re^{-4n} ,\]
for all but finitely many $n$, a.s. Thus
\[ F_S ( 2\lambda^{-1} \gamma_{n-1}^{-2} (  u \tau_{n-1}^2 -1 ) )
\geq F_S ( 2\lambda^{-1} u - 2 \lambda^{-1} u \re^{-4n} - 2 \lambda^{-1} \re^{-4n} ) 
\geq F_S ( 2\lambda^{-1} u ) - \re^{-2n} ,\]
provided $u \in [0,\re^n]$, using the fact that the density of $S$
given by~\eqref{levy-pdf} is uniformly bounded. Combined with~\eqref{stable-upper}
this completes the proof.
\end{proof}

The origin of the value $1/4$ in Proposition~\ref{prop:turning} is the following fact.

\begin{lemma}
\label{lem:quarter}
Let $S$ be a random variable with the distribution given by~\eqref{levy-cdf}, 
and let  $Z$ be a standard normal random variable, independent of $S$. Then
\[ \Pr ( Z \sqrt{S} > 1 ) = \frac{1}{4} .\]
\end{lemma}
\begin{proof}
Recall that $\Phi$ denotes the distribution function
of the standard normal distribution, and that $\bar\Phi (u) := 1 - \Phi (u)$.
Observe that
\[ p := \Pr ( Z >  S^{-1/2} ) = \Exp [ \Pr ( Z >   S^{-1/2} \mid S ) ] = \Exp \bar \Phi (  S^{-1/2}  ) .\]
Here $\bar \Phi : \R \to (0,1)$ is a strictly decreasing, continuous function
with $\lim_{x \to -\infty} \bar \Phi (x) =0$ and $\lim_{x \to +\infty} \bar \Phi(x) =1$, 
so $\bar \Phi$ has a strictly decreasing inverse $\bar\Phi^{-1} : (0,1) \to \R$, which is positive
on $(0,1/2)$ and negative on $(1/2,1)$, and
\begin{align}
\label{q-integral}
 p & = \int_0^1 \Pr ( \bar\Phi ( S^{-1/2}  ) \geq u ) \ud u = \int_0^{1/2} \Pr ( S \geq   ( \bar\Phi^{-1} (u) )^{-2}  ) \ud u \nonumber\\
& = \int_0^{1/2} \bar F_S (   ( \bar\Phi^{-1} (u) )^{-2}   ) \ud u ,\end{align}
where $\bar F_S (u ) := 1 - F_S(u)$ for $u \in \RP$. 
Now applying the formula in~\eqref{levy-cdf} we get $\bar F_S (   ( \bar\Phi^{-1} (u) )^{-2}   ) = 1 - 2 \bar \Phi ( \bar \Phi^{-1} (u) )$, so that
\[ p = \int_0^{1/2} ( 1 -2 u) \ud u = \frac{1}{4} ,\]
as claimed.
\end{proof}

Now we can complete the proof of Proposition~\ref{prop:turning}.

\begin{proof}[Proof of Proposition~\ref{prop:turning}.]
Given $\cF_{n-1}$, we have that at time $T_{n-1}$, $Q_{n-1} (X_{n-1}) = 0$
and the server then heads to $X_n = X_{n-1} + \eta_n$. At time $T_n$,
after serving the queue at $X_n$, the server inspects the queues at $X_n \pm 1$.

First we obtain an upper bound on $q_n$, the probability that the server changes direction.
We have
\[ q_n \leq \Pr ( Q_n (X_{n} - \eta_n ) \geq Q_n (X_n + \eta_n ) \mid \cF_{n-1} ) .\]
Note that $Q_{n-1} (X_n - \eta_n ) = Q_{n-1} (X_{n-1}) = 0$, i.e., the queue at $X_n - \eta_n$
was empty at time $T_{n-1}$. Thus $Q_n (X_n - \eta_n )$ is Poisson with mean $\lambda$ times $\tau_n = T_n - T_{n-1}$.
Set $\nu_n := Q_n ( X_n - \eta_n)$.

The queue at $X_n + \eta_n$,
which is neither $X_n$ nor $X_{n-1} = X_n - \eta_n$,
 was, prior to time $T_n$,
 last inspected by the server no more recently than at time $T_{n-2}$ (when the server decided to move to $X_{n-1}$).
Thus the number of customers at the queue at time $T_n$
is at least $\nu'_n + \nu''_n$ where
\begin{align*}
\nu'_n := Q_n (X_n + \eta_n ) - Q_{n-1} (X_n + \eta_n ) , \text{ and } \nu''_n := Q_{n-1} (X_n + \eta_n ) - Q_{n-2} (X_n + \eta_n ) .
\end{align*}
So we have
\[ q_n \leq \Pr ( \nu_n \geq \nu_n' + \nu_n''  \mid \cF_{n-1} ) , \]
where $\nu_n \sim P ( \lambda \tau_n )$ and $\nu_n' \sim P ( \lambda \tau_n)$
are both $\cF_{n}$-measurable, and are conditionally independent given $\tau_n$, and 
$\nu_n'' \sim P (\lambda \tau_{n-1})$ is $\cF_{n-1}$-measurable.
 Define
\begin{align*}
Z_n & = (\lambda \tau_n)^{-1/2} ( \nu_n - \lambda \tau_n); \\
Z'_n & =  (\lambda \tau_n)^{-1/2} ( \nu_n' - \lambda \tau_n); \\
Z''_n & = (\lambda \tau_{n-1})^{-1/2} ( \nu_n'' - \lambda \tau_{n-1}). \end{align*}
Then we get
\begin{align*}
q _n & \leq \Pr ( Z_n \geq Z_n' + ( \tau_{n-1} / \tau_n )^{1/2} Z_n'' + \lambda^{1/2} \tau_n^{-1/2} \tau_{n-1} \mid \cF_{n-1} ) .\end{align*}
Hence, writing $W_n := Z_n - Z_n'$, we have
\begin{align}
\label{qn0}
q_n & \leq \Pr ( W_n \geq \lambda^{1/2} \tau_n^{-1/2} \tau_{n-1}  - \re^{-n} \mid \cF_{n-1} )
+ \Pr ( | Z_n ''| (\tau_{n-1} / \tau_n )^{1/2} \geq \re^{-n} \mid \cF_{n-1} ) ,
\end{align} 
where, for $p \in (0,1/2)$,
\begin{align*}
\Pr ( | Z_n ''| (\tau_{n-1} / \tau_n )^{1/2} \geq \re^{-n} \mid \cF_{n-1} ) 
& \leq \Pr ( | Z_n'' | \geq \tau_{n-1}^{p} \mid \cF_{n-1} ) + \Pr ( \tau_{n} \leq \tau_{n-1}^{1+2p} \re^{2n} \mid \cF_{n-1}) .
\end{align*}
Since $Z_n''$ is $\cF_{n-1}$-measurable, we have that $\Pr ( | Z_n'' | \geq \tau_{n-1}^{p} \mid \cF_{n-1} ) = \1 { | Z_n'' | \geq \tau_{n-1}^{p} }$. We show that this event occurs only finitely often.
Set $\cF_{n-2}^+ := \sigma (\cF_{n-2}, \tau_{n-1} )$; note that
$\nu''$ (and hence $Z_n''$) only depend on $\cF^+_{n-2}$ via $\tau_{n-1}$.
We have that
\begin{align*}
\Pr ( | Z_n ''| \geq \tau_{n-1}^p \mid \cF_{n-2}^+ )
& \leq \Pr ( | Z_n''| \geq \tau_{n-1}^p, \tau_{n-1} \geq n \mid \cF_{n-2}^+ )
+ \1 { \tau_{n-1} \leq n } \\ 
& \leq \Pr ( | Z_n''| \geq \tau_{n-1}^p, \tau_{n-1} \geq n \mid \cF_{n-2}^+ ) ,\end{align*}
for all but finitely many $n$, a.s., by Proposition~\ref{prop:lower-bound}.
Now
\begin{align*}
 \Pr ( | Z_n''| \geq \tau_{n-1}^p, \tau_{n-1} \geq n \mid \cF_{n-2}^+ )
& = \Pr ( | \nu_n'' - \lambda \tau_{n-1} | \geq \lambda^{1/2} \tau_{n-1}^{(1/2)+p} , \tau_{n-1} \geq n \mid \cF_{n-2}^+ ) \\
& \leq \sup_{s \geq n } \Pr ( | P (\lambda s) - \lambda s | \geq \lambda^{1/2} s^{(1/2)+p} ) \\
& \leq \re^{-n^\eps} ,\end{align*}
for some $\eps>0$ and all $n$ sufficiently large, by Poisson concentration (see e.g.~\cite[p.~17]{penrose}).
So we conclude that $\sum_{n \geq 2} \Pr ( | Z_n ''| \geq \tau_{n-1}^p \mid \cF_{n-2}^+ ) < \infty$, a.s.,
and, since $\{ | Z_n ''| \geq \tau_{n-1}^p \} \in \cF_{n-1}^+$,
L\'evy's extension of the Borel--Cantelli lemma implies that
$|Z_n''| < \tau_{n-1}^p$ for all but finitely many $n$, a.s.
On the other hand, we have from Proposition~\ref{prop:lower-bound} that a.s., for all but finitely many $n$,
$\tau_{n-1}^{1+2p} \re^{2n} \leq \tau_{n-1}^\alpha$,
provided that we choose $\alpha \in (1+2p,2)$. Hence by Lemma~\ref{lem:tau-growth} we have for some $\eps>0$, a.s.,
\[ \Pr ( \tau_{n} \leq \tau_{n-1}^{1+2p} \re^{2n} \mid \cF_{n-1}) \leq \Pr ( \tau_{n} \leq \tau_{n-1}^\alpha \mid \cF_{n-1} ) \leq \re^{-n^\eps} ,\]
for all $n$ sufficiently large. Combining these estimates we get that for some $\eps>0$,
a.s., for all but finitely many $n$,
\begin{equation}
\label{Z''}
\Pr ( | Z_n ''| (\tau_{n-1} / \tau_n )^{1/2} \geq \re^{-n} \mid \cF_{n-1} ) \leq \re^{-n^\eps} .
\end{equation}
Thus from~\eqref{qn0} with~\eqref{Z''} we see that
\begin{equation}
\label{qn1}
 q_n \leq  \Pr ( W_n \geq    \lambda^{1/2} \tau_n^{-1/2} \tau_{n-1} - \re^{-n} \mid \cF_{n-1} )
+  \re^{-n^\eps} , \as,\end{equation}
for all but finitely many $n$.

For $u \in \R$ write $G_n (u) := \Pr (W_n \geq u \mid \tau_n)$. 
Set $S_n := \tau_n / \tau_{n-1}^2$ and, as above, set $\cF_{n-1}^+ := \sigma ( \cF_{n-1}, \tau_n)$. Then we have
\begin{align*}
\Pr ( W_n \geq \lambda^{1/2} S_n^{-1/2} - \re^{-n} \mid \cF_{n-1} )
= \Exp \big( \Pr ( W_n \geq \lambda^{1/2} S_n^{-1/2} - \re^{-n} \mid \cF^+_{n-1} ) \bigmid \cF_{n-1} \big) .\end{align*}
Since $W_n$ depends on $\cF^+_{n-1}$ only through $\tau_n$, and  $S_n$ is $\cF_{n-1}^+$-measurable, we have
\[ \Pr ( W_n \geq \lambda^{1/2} S_n^{-1/2} - \re^{-n} \mid \cF^+_{n-1} )  = G_n ( \lambda^{1/2} S_n^{-1/2} -\re^{-n} ) .\]
Since $W_n = (\lambda \tau_n)^{-1/2} (\nu_n - \nu_n')$ we have from Lemma~\ref{lem:poisson} and the fact that $\tau_n \geq 1$ that
\[ G_n ( u ) \leq \bar \Phi ( u /\sqrt{2} ) + C \tau_n^{-1/2} \log ( 1 +\tau_n ), \as,  \]
for  all $n$ and all $u \in \R$.
It follows that
\begin{align}
\label{qn5}
 \Pr ( W_n \geq \lambda^{1/2} S_n^{-1/2} - \re^{-n} \mid \cF_{n-1} ) & = \Exp ( G_n ( \lambda^{1/2} S_n^{-1/2} - \re^{-n} ) \mid \cF_{n-1} ) 
 \nonumber\\
& \leq \Exp ( \bar \Phi ( \lambda^{1/2} S_n^{-1/2} /\sqrt{2} - \re^{-n} ) \mid \cF_{n-1} ) 
\nonumber\\
& {} \qquad {}
 + C \Exp ( \tau_n^{-1/2} \log ( 1 +\tau_n ) \mid \cF_{n-1} ) ,
\end{align}
for all $n$ sufficiently large.
Here we have, since $\tau_n \geq 1$, for $\alpha \in (1,2)$, a.s.,
\begin{align}
\label{qn6}
 \Exp ( \tau_n^{-1/2} \log ( 1 +\tau_n ) \mid \cF_{n-1} ) & \leq \Pr ( \tau_n \leq \tau_{n-1}^\alpha \mid \cF_{n-1} ) + \tau_{n-1}^{-\alpha/3}   \leq \re^{-n^\eps} + \re^{-n} ,\end{align}
for all but finitely many $n$, by Lemma~\ref{lem:tau-growth} and Proposition~\ref{prop:lower-bound}. 
Moreover, since the standard normal density is uniformly bounded, we have that for some $C < \infty$ and all $u \in \R$,
$\bar \Phi (  u - \re^{-n} ) \leq \bar \Phi (u) + C \re^{-n}$. 
Thus, for some $\eps>0$,
\begin{equation}
\label{qn2}
 \Pr ( W_n \geq \lambda^{1/2} S_n^{-1/2} - \re^{-n} \mid \cF_{n-1} ) \leq \Exp ( \bar \Phi ( \lambda^{1/2} S_n^{-1/2} /\sqrt{2} ) \mid \cF_{n-1} )  + \re^{-n^\eps} , \as, \end{equation}
for all but finitely many $n$.

Similarly to~\eqref{q-integral}, we have
\begin{align*}
\Exp ( \bar \Phi   ( \lambda^{1/2} S_n^{-1/2} / \sqrt{2}) \mid \cF_{n-1} )
& = \int_0^{1/2} \Pr ( S_n \geq \lambda (\bar \Phi^{-1} (u))^{-2} / 2 \mid \cF_{n-1} ) \ud u  .\end{align*}
Set $a_n := \bar \Phi ( ({\lambda/2})^{1/2} \re^{-n/2} )$; then $a_n \in (0,1/2)$ with $a_n \to 1/2$, and
$\lambda (\bar \Phi^{-1} (u))^{-2} / 2 \in [0,\re^n]$ for $u \in (0,a_n)$. Thus, by Lemma~\ref{lem:stable},
\begin{align}
\label{qn3}
\Exp ( \bar \Phi   ( \lambda^{1/2} S_n^{-1/2} / \sqrt{2}) \mid \cF_{n-1} )
& \leq \int_0^{a_n} \left( \bar F_S ( ( \bar \Phi^{-1} (u))^{-2} ) + \re^{-2n} \right) \ud u  + \frac{1}{2} - a_n \nonumber\\
& \leq q + \re^{-2n}  + C \re^{-n/2}, 
\end{align}
by~\eqref{q-integral} and the fact that the standard normal density is uniformly bounded;
here $q=1/4$ is the probability in~\eqref{q-integral} and Lemma~\ref{lem:quarter}.
Combining~\eqref{qn1}, \eqref{qn2}, and \eqref{qn3} we obtain $q_n \leq q + \re^{-n^\eps}$ for all but finitely many $n$, a.s.

Now we obtain a lower bound on $q_n$. In addition to $\nu_n, \nu_n', \nu_n''$ defined above,
also define $\nu_n''' := Q_{n-2} (X_n + \eta_n)$. Then $\nu_n'''$ is $\cF_{n-1}$-measurable.
With $W_n$ and $Z_n''$ as defined above, we have
\begin{align*}
q_n & \geq \Pr ( Q_n (X_n - \eta_n ) > Q_n ( X_n + \eta_n ) \mid \cF_{n-1} ) \\
& = \Pr ( \nu_n > \nu_n' + \nu_n'' + \nu_n''' \mid \cF_{n-1} ) \\
& = \Pr (W_n > (\tau_{n-1} / \tau_n)^{1/2} Z_n'' + \lambda^{1/2} \tau_n^{-1/2} \tau_{n-1}
+ \lambda^{-1/2} \tau_n^{-1/2} \nu_n''' \mid \cF_{n-1} ) \\
& \geq \Pr (W_n > \lambda^{1/2} \tau_n^{-1/2} \tau_{n-1} + \re^{-n} + \lambda^{-1/2} \tau_n^{-1/2} \nu_n''' \mid \cF_{n-1} ) \\
& {} \qquad {} - \Pr ((\tau_{n-1} / \tau_n)^{1/2} |Z_n''| \geq \re^{-n} \mid \cF_{n-1} ).
\end{align*}
Applying~\eqref{Z''}, we see that a.s., for all but finitely many $n$,
\begin{align*}
 q_n & \geq \Pr (W_n > \lambda^{1/2} \tau_n^{-1/2} \tau_{n-1} + \re^{-n} + \lambda^{-1/2} \tau_n^{-1/2} \nu_n''' \mid \cF_{n-1} )  - \re^{-n^\eps} \\
& \geq \Pr ( W_n > \lambda^{1/2} \tau_n^{-1/2} \tau_{n-1} + 2 \re^{-n} \mid \cF_{n-1} )
- \Pr  ( \lambda^{-1/2} \tau_n^{-1/2} \nu_n''' \geq \re^{-n} \mid \cF_{n-1} ) - \re^{-n^\eps} .\end{align*}
Here we have that, a.s., for all but finitely many $n$,
\begin{align*}
\Pr  ( \lambda^{-1/2} \tau_n^{-1/2} \nu_n''' \geq \re^{-n} \mid \cF_{n-1} )
& \leq \Pr ( \nu_n''' \geq \lambda^{1/2} \tau_{n-1}^{3/4} \re^{-n} \mid \cF_{n-1} )
+ \Pr ( \tau_n \leq \tau_{n-1}^{3/2} \mid \cF_{n-1} ) \\
& \leq \1 { \nu_n''' \geq \lambda^{1/2} \tau_{n-1}^{3/4} \re^{-n} } + \re^{-n^\eps} ,\end{align*}
by Lemma~\ref{lem:tau-growth}.
By Lemma~\ref{lem:T-growth} and Lemma~\ref{lem:tau-growth} again, we have that a.s.,
for all but finitely many $n$, $T_{n-2} \leq 2 \tau_{n-2}$ and $\tau_{n-2} \leq \tau_{n-1}^{2/3}$,
so, by Proposition~\ref{prop:lower-bound}, since $\tau_{n-1}^{1/12} \re^{-n} \to \infty$,  
\[ \lambda^{1/2} \tau_{n-1}^{3/4} \re^{-n} = \lambda^{1/2} \tau_{n-1}^{1/12} \re^{-n} \tau_{n-1}^{2/3} \geq 2 \lambda T_{n-2} .\]
Thus, a.s., for all but finitely many $n$,
\begin{equation}
\label{qn4}
 \1 { \nu_n''' \geq \lambda^{1/2} \tau_{n-1}^{3/4} \re^{-n} } \leq \1 { \nu_n''' \geq 2 \lambda T_{n-2} } .\end{equation}
The queue at $X_n + \eta_n$,
which is neither $X_n$ nor $X_{n-1} = X_n - \eta_n$,
 was, prior to time $T_n$,
 last inspected by the server no more recently than at time $T_{n-2}$, at which point the server decided to move to $X_{n-1}$
(and not $X_n +\eta_n$). 
Thus $\nu_n'''$ is stochastically dominated by $P( \lambda T_{n-2} )$,
so 
\begin{align*}
\Pr ( \nu_n''' \geq 2 \lambda T_{n-2} ) & \leq \Pr ( P (\lambda T_{n-2} ) \geq 2 \lambda T_{n-2} ) \\
& \leq \sup_{s \geq n-2} \Pr ( P ( \lambda s) \geq 2 \lambda s ) ,\end{align*}
using the fact that $T_{n-2} \geq n-2$, a.s. Then by standard Poisson tail bounds (see e.g.~\cite[p.~17]{penrose})
we have that this last quantity is bounded by $\re^{-\delta n}$ for some $\delta >0$ and all $n$ sufficiently
large. Hence the Borel--Cantelli lemma shows that the indicator
random variable in~\eqref{qn4} is a.s.~equal to $0$ for all but finitely many~$n$.
Thus, a.s., for all but finitely many $n$,
\[ \Pr  ( \lambda^{-1/2} \tau_n^{-1/2} \nu_n''' \geq \re^{-n} \mid \cF_{n-1} ) \leq  \re^{-n^\eps} .\]
It follows that, for some $\eps>0$, a.s., for all but finitely many $n$,
\[ q_n     \geq \Pr ( W_n > \lambda^{1/2} \tau_n^{-1/2} \tau_{n-1} + 2 \re^{-n} \mid \cF_{n-1} )  - \re^{-n^\eps}.\]
The estimation of the main term here proceeds in a similar way to in the upper bound.
Similarly to~\eqref{qn5} and~\eqref{qn6}, we have that
\begin{align*} \Pr ( W_n > \lambda^{1/2} \tau_n^{-1/2} \tau_{n-1} + 2 \re^{-n} \mid \cF_{n-1} )
& \geq \Exp ( \bar \Phi ( \lambda^{1/2} S_n^{-1/2} / \sqrt{2} + 2 \re^{-n} ) \mid \cF_{n-1} ) - \re^{-n^\eps } \\
& \geq \Exp ( \bar \Phi ( \lambda^{1/2} S_n^{-1/2} / \sqrt{2} ) \mid \cF_{n-1} ) - C \re^{-n} - \re^{-n^\eps } 
.\end{align*}
 Finally, similarly to~\eqref{qn3}, we have
\begin{align*}
 \Exp ( \bar \Phi ( \lambda^{1/2} S_n^{-1/2} / \sqrt{2} ) \mid \cF_{n-1} )
\geq \int_0^{a_n} ( \bar F_S ( ( \bar \Phi^{-1} (u))^{-2} ) - \re^{-2n} )
\geq q - \re^{-2n} - C \re^{-n/2} ,\end{align*}
and this gives $q_n \geq q - \re^{-n^\eps}$, as required.
\end{proof}

\section{Proofs of theorems}
\label{sec:proofs}

With $q = 1/4$ as appearing in Proposition~\ref{prop:turning}, set
$a :=  \frac{1-2q}{q} = 2$.
To prove Theorem~\ref{thm:recurrent} we consider the function defined for $x \in \Z$ and $i \in \{-1,+1\}$ by
\[ f(x,i) := x + a \1{ i = 1}   . \]
We consider $Y_n := f(X_n, \eta_n)$; recall that $(X_n, \eta_n)$ is $\cF_{n-1}$-measurable.
Note that, for all $n \in \ZP$, $|X_n - Y_n | \leq a$.
The next result describes the increments of $Y_n$, and, in particular, 
shows that it is close to a martingale.

\begin{lemma}
\label{lem:lyapunov}
Let  $q_n$ be the $\cF_{n-1}$-measurable random variable defined in~\eqref{q-def}.
\begin{itemize}
\item[(a)] We have that, for all $n \geq 0$,
\begin{equation}
\label{Ybound}
| Y_{n+1} - Y_n | \leq 3, \as \end{equation}
\item[(b)] There is a  sequence $\delta_n$ of non-negative $\cF_{n-1}$-adapted random variables such that,
\begin{equation}
\label{Y-mart}
 \left| \Exp ( Y_{n+1} - Y_n \mid \cF_{n-1} )\right| \leq \delta_n, \as, \end{equation}
for all $n \geq 1$, and, for some $\eps>0$, $\delta_n \leq \re^{-n^\eps}$ for all but finitely many $n$, a.s. In particular, $\sum_{n \geq 1} \delta_n < \infty$, a.s.
\item[(c)] 
We have 
\begin{equation}
\label{Y-second}
 \Exp ( (Y_{n+1} - Y_n)^2 \mid \cF_{n-1} ) = 1 + 8 q_n, \as 
\end{equation}
\end{itemize}
\end{lemma}
\begin{proof}
For $x \in \Z$ and $i \in \{-1,+1\}$, define
\begin{align*}
\Delta^+ (x,i) := f(x+i, i) - f(x,i), \text{ and } \Delta^- (x,i) := f (x-i, -i) - f(x,i) .\end{align*}
Then since $X_{n+1} = X_n + \eta_{n+1}$, we have that
\begin{equation}
\label{Y-inc} 
Y_{n+1} - Y_n  = \Delta^+ ( X_n, \eta_n ) \1{ \eta_{n+1} = \eta_n} + \Delta^- (X_n, \eta_n) \1{ \eta_{n+1} \neq \eta_n} .\end{equation}
Note that $\Delta^+ (x, i) =   i$
and
\[ \Delta^- (x, i) =  -i + a \1{i= -1} -a\1 { i =1} = -i -ai .\]
Thus from~\eqref{Y-inc} we have
$| Y_{n+1} - Y_n | = 1 + a \1{ \eta_{n+1} \neq \eta_n}   \leq 3$, a.s., giving~\eqref{Ybound}.
For $q_n$ the $\cF_{n-1}$-measurable random variable defined in~\eqref{q-def}, we have
from~\eqref{Y-inc} that
\[  \Exp ( Y_{n+1} - Y_n \mid \cF_{n-1} ) = (1- q_n) \Delta^+ (X_n , \eta_n) + q_n \Delta^- (X_n , \eta_n) .\]
Since $\Delta^\pm (x,i)$ are uniformly bounded, we have from Proposition~\ref{prop:turning} that
there is an $\cF_{n-1}$-adapted sequence $\eps_n$ with $\delta_n := | \eps_n | \leq \re^{-n^\eps}$
for all but finitely many $n$, such that
\begin{align*}
(1-q_n) \Delta^+ (x,i) + q_n \Delta^- (x,i) = ( 1- q) \Delta^+ (x,i) + q \Delta^- (x,i) + \eps_n .\end{align*}
Here we have that
\begin{align*}
  ( 1- q) \Delta^+ (x,i) + q \Delta^- (x,i) 
& = (1-q) i + q ( -i - ai) \\
& = (1-2q)i - a q i =0 ,
\end{align*}
for all $x$ and all $i$, by choice of $a$. This gives~\eqref{Y-mart}.

For the second moment, note that, by~\eqref{Y-inc}, 
\[ \Exp ( (Y_{n+1} - Y_n)^2 \mid \cF_{n-1} ) = (1-q_n) ( \Delta^+ ( X_n, \eta_n))^2 + q_n ( \Delta^- ( X_n, \eta_n))^2 .\]
Here $(\Delta^+ (x,i))^2 = 1$ and $(\Delta^- (x,i))^2 = (1+a)^2 = 9$, and~\eqref{Y-second} follows.
\end{proof}

The proofs of our two main theorems will use the following martingale decomposition.
Set $\theta_n := \Exp(  Y_{n+1} - Y_n \mid \cF_{n-1} )$ for $n \in \N$.
Note that, by~\eqref{Ybound}, $|\theta_n| \leq 3$, a.s.
As in Doob's decomposition, for $n \geq 1$ let $A_n := \sum_{i=1}^{n-1} \theta_i$, and set $M_n := Y_n - A_n$,
so that
\[ \Exp ( M_{n+1} - M_n \mid \cF_{n-1} ) = \Exp ( Y_{n+1} - Y_n \mid \cF_{n-1} ) - (A_{n+1} -A_n ) = \theta_n - \theta_n = 0 .\]
Thus $M_n$ is an $\cF_{n-1}$-adapted martingale ($n \geq 1$), and
\begin{equation}
\label{M-bound}
| M_{n+1} - M_n | \leq | Y_{n+1} - Y_n | + | \theta _n | \leq 6 , \as 
\end{equation}
Note that $| \theta_n | \leq \delta_n$, a.s., where $\sum_{n \geq 1} \delta_n < \infty$, a.s., by Lemma~\ref{lem:lyapunov}(b).

\begin{proof}[Proof of Theorem~\ref{thm:recurrent}.]
We have that $M_n$ is a martingale with uniformly bounded increments, by~\eqref{M-bound}.
 Theorem~5.3.1 of~\cite{durrett} says that $M_n$
 either oscillates ($\liminf_{n \to \infty} M_n = -\infty$, $\limsup_{n \to \infty} M_n = +\infty$),
or converges ($\lim_{n \to \infty} M_n \to M_\infty \in \R$).
Suppose that $M_n \to M_\infty$. Then
\[ \limsup_{n \to \infty} | Y_n | \leq |M_\infty |+ \limsup_{n \to \infty} \sum_{i=1}^{n-1} |\theta_i|
\leq |M_\infty| + \sum_{i=1}^\infty \delta_i < \infty , \as \]
But since $|Y_n | \geq |X_n| - a$, this contradicts Kurkova and Menshikov's result~\eqref{eq:limsup}, which says that $\limsup_{n \to \infty} | X_n | = \infty$, a.s.
Thus we must have that $M_n$ oscillates, a.s.
 Then
since $\sup_{n \geq 1} |A_n| < \infty$, it follows that $Y_n$ oscillates,  and hence $X_n$ oscillates. Since $X_n \in \Z$
satisfies $| X_{n+1} - X_n | =1$, and it oscillates between $-\infty$ and $+\infty$, we must have $X_n = x$ i.o.~for any $x \in \Z$.
Hence, a.s., for every $x \in \Z$, the set $\{ t \geq 0 : S(t) = x \}$ is unbounded. The result extends to all $x \in \R$ by continuity of the
server's trajectory.
\end{proof}

Now we turn to the proof of Theorem~\ref{thm:growth}.
The next result is essentially an inversion
of Propositions~\ref{prop:lower-bound} and~\ref{prop:upper-bound}.

\begin{lemma}
\label{lem:Nt}
We have that
\[ \lim_{t \to \infty} \frac{N_t}{\log \log t} = \frac{1}{\log 2}, \as \]
\end{lemma}
\begin{proof}
Let $\alpha \in (1,2)$.
Since $N_t \to \infty$ a.s., we have from Proposition~\ref{prop:lower-bound}
that a.s., for all $t$ sufficiently large,
\[ t \geq T_{N_t} \geq \re^{\alpha^{N_t}} .\]
It follows that $\log \log t \geq N_t \log \alpha$ for all $t$ sufficiently large.
Hence
\[ \limsup_{t \to \infty} \frac{N_t}{\log \log t} \leq \frac{1}{\log \alpha}, \as \]
Since $\alpha \in (1,2)$ was arbitrary, we get  
\[ \limsup_{t \to \infty} \frac{N_t}{\log \log t} \leq \frac{1}{\log 2}, \as \]
On the other hand, for $\beta >2$ we have from Proposition~\ref{prop:upper-bound} that a.s.,
for all $t$ sufficiently large,
\[ t \leq T_{N_t+1} \leq \re^{\beta^{N_t+1}} .\]
It follows that $\log \log t \leq (N_t + 1) \log \beta$ for all $t$ sufficiently large.
Hence
\[   \liminf_{t \to \infty} \frac{N_t}{\log \log t} \geq \frac{1}{\log \beta}, \as \]
Since $\beta >2$ was arbitrary, we get  
\[ \liminf_{t \to \infty} \frac{N_t}{\log \log t} \geq \frac{1}{\log 2}, \as \]
Combined with the $\limsup$ result, this gives the statement in the lemma.
\end{proof}

Next we have an iterated logarithm law for $X_n$.

\begin{lemma}
\label{lem:X}
We have that
\[ \limsup_{n \to \infty} \frac{\pm X_n}{\sqrt{6 n \log \log n}} = 1, \as \]
\end{lemma}
\begin{proof}
First note that
\begin{align*}
 \Exp ( (M_{n+1} -M_n)^2 \mid \cF_{n-1} )  
& = \Exp ( (Y_{n+1} - Y_n -\theta_n)^2 \mid \cF_{n-1} ) \\
& = \Exp ( (Y_{n+1} - Y_n)^2 \mid \cF_{n-1} ) - \theta_n^2,\end{align*}
where, by Lemma~\ref{lem:lyapunov}, $|\theta_n | \leq \delta_n$ a.s., and both $\delta_n$
and $\delta_n^2$ are a.s.~summable.
Thus from~\eqref{Y-second} and Proposition~\ref{prop:turning} we have that
\[  \Exp ( ( M_{n+1} - M_n)^2 \mid \cF_{n-1} ) = 1 + 8 q +\eps_n, \as, \]
where $\eps := \sum_{n \geq 1} \eps_n$ has $|\eps| < \infty$, a.s.
Since $q=1/4$, it follows that
\begin{equation}
\label{sn}
 s^2_n := \sum_{i=1}^n \Exp ( ( M_{i+1} - M_i)^2 \mid \cF_{i-1} ) = 3 n + \eps + o(1), \as\end{equation}
The conditions~\eqref{M-bound} and~\eqref{sn} show that $M_n$ and $-M_n$
each satisfy the martingale law of the iterated logarithm~\cite{stout}, yielding
\[ \limsup_{n \to \infty} \frac{\pm M_n}{\sqrt{6 n \log \log n}} = 1, \as \]
Since $|X_n - Y_n | \leq a$ and $|\theta_n| \leq \delta_n$, we have that
\[ \sup_{n} | M_n - X_n | \leq a+ \sup_{n} | M_n - Y_n | \leq a + \sum_{n=1}^\infty \delta_n < \infty, \as \]
Thus the iterated logarithm law for $M_n$ transfers to $X_n$.
\end{proof}

\begin{proof}[Proof of Theorem~\ref{thm:growth}.]
Since $N_t \to \infty$ as $t \to \infty$,  Lemma~\ref{lem:X} shows that
\[ \limsup_{t \to \infty} \frac{\pm X_{N_t}}{\sqrt{N_t \log \log N_t}} = \sqrt{6}, \as \]
Combining this with Lemma~\ref{lem:Nt} gives
\[  \limsup_{t \to \infty} \frac{\pm X_{N_t}}{\sqrt{\log \log t \log \log \log \log t}} = \sqrt{\frac{6}{\log 2}}, \as \]
We have from~\eqref{eq:S} and the fact that $|X_{n+1}-X_n| = 1$
that 
\[ S(t) \geq \min \{ X_{N_t} , X_{N_t+1} \} \geq X_{N_t} -1, \text{ and } S(t) \leq \max \{ X_{N_t} , X_{N_t+1} \} \leq X_{N_t} + 1. \]
The result follows.
\end{proof}

\appendix

\section{Auxiliary lemmas}
\label{sec:aux}

Recall that $\Phi$ denotes the distribution function of the standard normal distribution.

\begin{lemma}
\label{lem:poisson}
Let $\kappa \geq 0$ and let $\nu \sim P (\kappa)$ and $\nu' \sim P (\kappa)$ be independent. Then
there exists $C \in \RP$ such that, for all $\kappa > 0$,
\[ \sup_{u \in \R} \left| \Pr ( \kappa^{-1/2} ( \nu - \nu' ) \leq u ) - \Phi ( u/\sqrt{2} ) \right| \leq C (1+\kappa)^{-1/2} \log (2+ \kappa) .\]
\end{lemma}
\begin{proof}
Let $u_\kappa := \kappa^{1/2} \log \kappa$.
Then, by symmetry,
\begin{align*} 
\sup_{u : |u| > u_\kappa} \left| \Pr ( \kappa^{-1/2} ( \nu - \nu' ) \leq u ) - \Phi ( u/\sqrt{2} ) \right|
& 
= \sup_{u : u > u_\kappa} \left| \Pr ( \kappa^{-1/2} ( \nu - \nu' ) > u ) - \bar \Phi ( u/\sqrt{2} ) \right| \\
& \leq  \Pr ( \kappa^{-1/2} ( \nu - \nu' ) > u_\kappa ) +  \bar \Phi ( u_\kappa/\sqrt{2} )
.\end{align*}
Here we have from standard Gaussian tail bounds (see e.g.~Theorem~1.2.3 of~\cite{durrett}) that
$\bar \Phi ( u_\kappa/\sqrt{2} ) = O ( \re^{-\kappa} )$, say, while, since $\nu' \geq 0$,
\[ \Pr ( \kappa^{-1/2} ( \nu - \nu' ) > u_\kappa ) \leq
 \Pr ( \nu > \kappa^{1/2} u_\kappa )  = O (\re^{-\kappa} ) , \]
by Poisson large deviations bounds (see e.g.~\cite[p.~17]{penrose}).
The result in the lemma will thus follow from the claim
that there exists $C \in \RP$ for which
\begin{equation}
\label{claim}
\sup_{ u : |u| \leq u_\kappa} \left| \Pr ( \kappa^{-1/2} ( \nu - \nu' ) \leq u ) - \Phi ( u/\sqrt{2} ) \right|
\leq C \kappa^{-1/2} \log \kappa  ,\end{equation}
for all $\kappa \geq 2$.
It remains to prove~\eqref{claim}.

By Poisson additivity, we can write
\[ \nu - \nu' =\sum_{j = 1}^{\lfloor \kappa \rfloor} \xi_j + \gamma - \gamma' ,\]
where $\gamma, \gamma', \xi_1, \xi_2,\ldots$ are independent  random variables,
each $\xi_j$ being the difference of two independent $P(1)$ random variables, 
and $\gamma, \gamma'$ being Poisson with mean $\kappa - \lfloor \kappa \rfloor < 1$.
Let $S_\kappa := \sum_{j=1}^{\lfloor \kappa \rfloor} \xi_j$.
Then $\Exp \xi_j = 0$, $\Exp ( \xi_j^2 ) = 2$, and $\Exp ( |\xi_j|^3 ) < \infty$, so the Berry--Esseen theorem (see e.g.~Theorem~3.4.9 of~\cite[p.~137]{durrett}) implies that
\begin{equation}
\label{berry}
 \sup_{u \in \R} \left| \Pr \left( {\lfloor \kappa \rfloor} ^{-1/2} S_\kappa 
\leq u \right)
- \Phi ( u /\sqrt{2} ) \right | \leq C \kappa^{-1/2} ,\end{equation}
for all $\kappa \geq 1$. 

First we prove one half of~\eqref{claim}. Since $\gamma \geq 0$,
\begin{align*}
\Pr ( \nu - \nu ' \leq u \kappa^{1/2} ) & \leq \Pr ( S_\kappa - \gamma' \leq u \kappa^{1/2} ) \\
& \leq \Pr ( S_\kappa \leq u \kappa^{1/2} + r ) + \Pr (\gamma' \geq r ) ,
\end{align*}
for any $r >0$.
Here we have that 
\begin{align*}
\Pr ( S_\kappa \leq u \kappa^{1/2} + r ) & = \Pr \left( \lfloor \kappa \rfloor^{-1/2} S_\kappa
\leq \left( \frac{\kappa}{\lfloor \kappa \rfloor} \right)^{1/2} u + \lfloor \kappa \rfloor^{-1/2} r \right) .\end{align*}
Note that, by Taylor's theorem,
\[ \left( \frac{\kappa}{\lfloor \kappa \rfloor} \right)^{1/2} \leq 
 \left( \frac{\lfloor \kappa \rfloor +1}{\lfloor \kappa \rfloor} \right)^{1/2} 
\leq 1 + C \kappa^{-1} ,\]
for some $C \in \RP$ and all $\kappa \geq 1$. Thus, by~\eqref{berry}
and the fact that the standard normal density is uniformly bounded,
\begin{align*}
\Pr ( S_\kappa \leq u \kappa^{1/2} + r ) & \leq \Phi ( u /\sqrt{2} )
+ C \kappa^{-1} |u| + C \kappa^{-1/2} r + C \kappa^{-1/2} ,
\end{align*}
for all $u$ and all $\kappa \geq 1$.
In particular, taking $r = \log \kappa$ we have
\[ \Pr ( S_\kappa \leq u \kappa^{1/2} + \log \kappa )  \leq \Phi ( u /\sqrt{2} )
+ C \kappa^{-1/2} \log \kappa , \text{ for all } u \in [-u_\kappa, u_\kappa ] ,\]
where $C< \infty$ does not depend on $u$ or $\kappa$.
On the other hand,
$\Pr (\gamma' \geq \log \kappa ) \leq \Pr ( P (1) \geq \log \kappa) = O (\kappa^{-1})$
by Poisson large deviations bounds (see e.g.~\cite[p.~17]{penrose}). This establishes one half of~\eqref{claim}.

For the other direction, we have that
\begin{align*}
\Pr ( \nu - \nu ' \leq u \kappa^{1/2} ) & \geq \Pr ( S_\kappa + \gamma \leq u \kappa^{1/2} ) \\
& \geq \Pr ( S_\kappa \leq u \kappa^{1/2} - r ) - \Pr (\gamma \geq r ) ,
\end{align*}
Taking $r = \log \kappa$ we get $\Pr (\gamma \geq \log \kappa ) = O (\kappa^{-1})$ and,
similarly to above, we get
\[ \Pr ( S_\kappa \leq u \kappa^{1/2} - \log \kappa ) \geq  \Phi ( u /\sqrt{2} )
- C \kappa^{-1/2} \log \kappa , \text{ for all } u \in [-u_\kappa, u_\kappa ] ,\] 
completing the proof of~\eqref{claim}.
\end{proof}

Finally, we record the following elementary result.

\begin{lemma}
\label{lem:reduced}
Let $X, Y$ be   random variables. Then for any $x \in \R$, 
\[ \Pr (X > x \mid X \leq Y) \leq \Pr (X > x). \]
\end{lemma}
\begin{proof}
For $x, y \in \R$ we have
\[ \Pr (X \leq x \mid X \leq y) = \frac{\Pr (X \leq \min \{ x , y\})}{\Pr (X \leq y)} = \begin{cases} 
1 & \text{if } x \geq y ,\\
\frac{\Pr (X \leq x)}{\Pr ( X \leq y)} & \text{if } x \leq y .\end{cases}
\]
In any case, we have $\Pr (X \leq x \mid X \leq y) \geq \Pr (X \leq x)$, and the result follows.
\end{proof}

\section*{Acknowledgements}

Some of this work was carried out while the first author was supported by a Grey Fellowship
at Durham University; the hospitality of Grey College and the Department of Mathematical Sciences
at Durham University
is gratefully acknowledged. The first author was also partially supported by the Heilbronn Institute
for Mathematical Research.

\end{document}